\documentclass{amsart}
\usepackage{amssymb, enumerate}
\usepackage{amsrefs}
\usepackage{todonotes}

\theoremstyle{plain}
\newtheorem{theorem}{Theorem}
\newtheorem*{thmslit}{Theorem \ref{slit}}
\newtheorem*{thmnonslit}{Theorem \ref{nonslit}}
\newtheorem*{thmJohnprop}{Theorem \ref{Johnprop}}
\newtheorem{lemma}[theorem]{Lemma}

\theoremstyle{definition}
\newtheorem{definition}[theorem]{Definition}

\theoremstyle{remark}

\newtheorem{question}{Question}

\newcommand{\lp}{\left(}

\newcommand{\rp}{\right)}
\newcommand{\e}{\epsilon}
\newcommand{\R}{\mathbb{R}}
\newcommand{\C}{\mathbb{C}}

\newcommand{\Hy}{\mathbb{H}}

\newcommand{\D}{\mathbb{D}}

\newcommand{\Z}{\mathbb{Z}}

\newcommand{\N}{\mathbb{N}}
\newcommand{\cl}{\overline}
\newcommand{\ti}{\textit}

\let\Re\relax
\let\Im\relax
\let\mod\relax

\DeclareMathOperator{\dist}{\textup{\text{dist}}}
\DeclareMathOperator{\diam}{\textup{\text{diam}}}
\DeclareMathOperator{\id}{\textup{\text{id}}}

\DeclareMathOperator{\supp}{\textup{\text{supp}}}

\DeclareMathOperator{\hcap}{\textup{\text{hcap}}}

\DeclareMathOperator{\mod}{\textup{\text{mod}}}
\DeclareMathOperator{\SLE}{\textup{\text{SLE}}}
\DeclareMathOperator{\Re}{\textup{\text{Re}}}
\DeclareMathOperator{\Im}{\textup{\text{Im}}}
\DeclareMathOperator{\length}{\textup{\text{length}}}
\DeclareMathOperator{\Area}{\textup{\text{Area}}}
\DeclareMathOperator{\Lip}{\textup{\text{Lip}}}

\DeclareMathOperator{\rad}{\textup{\text{rad}}}

\numberwithin{equation}{section}
\numberwithin{theorem}{section}

\begin{document}

\title{Loewner chains and H\"older geometry}
\author{Kyle Kinneberg}
\address{Department of Mathematics, Rice University, 6100 Main St., Houston TX 77005}
\email{kyle.kinneberg@rice.edu}

\date{\today}

\maketitle

\begin{abstract}
The Loewner equation provides a correspondence between continuous real-valued functions $\lambda_t$ and certain increasing families of half-plane hulls $K_t$. In this paper we study the deterministic relationship between specific analytic properties of $\lambda_t$ and geometric properties of $K_t$. Our motivation comes, however, from the stochastic Loewner equation ($\SLE_{\kappa}$), where the associated function $\lambda_t$ is a scaled Brownian motion and the corresponding domains $\Hy \backslash K_t$ are H\"older domains. We prove that if the increasing family $K_t$ is generated by a simple curve and the final domain $\Hy \backslash K_T$ is a H\"older domain, then the corresponding driving function has a modulus of continuity similar to that of Brownian motion. Informally, this is a converse to the fact that $\SLE_{\kappa}$ curves are simple and their complementary domains are H\"older, when $\kappa < 4$. We also study a similar question outside of the simple curve setting, which informally corresponds to the $\SLE$ regime $\kappa > 4$. In the process, we establish general geometric criteria that guarantee that $K_t$ has a $\Lip(1/2)$ driving function.
\end{abstract}

\section{Introduction} \label{intro}

Let $\Hy = \{z \in \C : \Im(z) > 0\}$ denote the upper half-plane in $\C$. A set $K \subset \Hy$ is a \ti{half-plane hull} if it is bounded and $\Hy \backslash K$ is a simply connected domain, i.e., is simply connected, connected, and open. The Loewner differential equation provides a correspondence between certain one-parameter families of growing hulls, $K_t$, and continuous real-valued functions $\lambda_t$. Under this correspondence, the (properly normalized) Riemann maps $g_t \colon \Hy \backslash K_t \rightarrow \Hy$ solve an ODE that is ``driven" by the function $\lambda_t$. An important point of investigation is the relationship between geometric properties of these hulls and analytic properties of the associated driving function, or driving term.

Classically, interest in the Loewner equation was motivated by questions about simply connected domains and their corresponding Riemann maps. For example, if one can view a given domain $\Omega = \Hy \backslash K$ as the final object in a continuously shrinking chain of domains $\Omega_t = \Hy \backslash K_t$, starting with $\Omega_0 = \Hy$, where the evolution is governed by a particular ODE, then one may hope to deduce properties of $\Omega$ from properties of the ODE. In this way, the driving term becomes a tool for analyzing the evolution process.

More recently, however, interest in Loewner's equation has been renewed by reversing this correspondence. Given a continuous function $\lambda_t$, one can solve the associated ODE to obtain Riemann maps $g_t$, corresponding to domains $\Hy \backslash K_t$ which shrink continuously in $t$. An interesting \ti{stochastic} process arises when $\lambda_t$ is equal to a Brownian motion, scaled by a parameter $\kappa>0$. In this case, the corresponding family of hulls is described by a random curve in $\Hy$, called an $\SLE_{\kappa}$ curve. For various values of $\kappa$, these curves have been shown (or are conjectured) to arise as scaling limits of certain discrete random planar processes.

The development of the $\SLE$ theory has rejuvenated interest in the deterministic theory as well. More specifically, the curves arising from $\SLE_{\kappa}$ processes are examples of fractal sets in the plane (cf. \cite{RS05, Bef08}), and the corresponding driving function (Brownian motion) is highly non-smooth. In this spirit, there has been an effort to understand better the deterministic Loewner correspondence when the driving terms and the corresponding hulls are allowed to be non-smooth.

In the present paper, we take up this topic, motivated by specific properties of $\SLE$ curves and their stochastic driving terms. Despite the fact that both of these objects are non-smooth, each does enjoy some type of regularity: Brownian motion is in the class of weak-$\Lip(1/2)$ functions; and the $\SLE_{\kappa}$ domains $\Hy \backslash K_t$ are, almost surely, H\"older continuous conformal images of $\Hy$. Our goal here is to investigate the deterministic correspondence between driving functions that are weak-$\Lip(1/2)$ and families of half-plane hulls whose complementary domains are H\"older domains. To make this precise, we will need to discuss more background and some technical terminology.

\subsection{Background}

Let $K \subset \Hy$ be a half-plane hull, so that $\Hy \backslash K$ is a simply connected domain. There is a unique conformal map $f \colon \Hy \rightarrow \Hy \backslash K$ with the expansion
\begin{equation} \label{hydrof}
f(z) = z - \frac{a_1}{z} +O \lp \frac{1}{z^2} \rp, \hspace{0.3cm} \text{as } |z| \rightarrow \infty
\end{equation}
where $a_1 \geq 0$ is a constant. The \ti{half-plane capacity} of $K$ is, by definition, 
$$\hcap(K) := a_1.$$ 
We will see later that this is a measure of the size of $K$, seen ``from infinity." As a basic example, we note that $\hcap(K) = 0$ if and only if $K=\emptyset$, in which case $f=\id_{\Hy}$. Moreover, $\hcap$ is monotonic, in the sense that $\hcap(K) \leq \hcap(K')$ whenever $K$ and $K'$ are hulls with $K \subset K'$. 

Observe that the inverse map $g = f^{-1} \colon \Hy \backslash K \rightarrow \Hy$ has the form
\begin{equation} \label{hydrog}
g(z) = z + \frac{a_1}{z} + O \lp \frac{1}{z^2} \rp, \hspace{0.3cm} \text{as } |z| \rightarrow \infty.
\end{equation}
In this paper, we frequently use the normalizations in \eqref{hydrof} and \eqref{hydrog}. Following standard terminology, we say that such maps $f$ and $g$ are \ti{hydrodynamically normalized}.

Let $\lambda_t$ be a continuous, real-valued function defined on a bounded interval $[0,T]$. The associated Loewner equation is
\begin{equation} \label{LE}
\partial_t g_t(z) = \frac{2}{g_t(z) - \lambda_t}, \hspace{0.4cm} g_0(z) = z \in \Hy.
\end{equation}
The solutions to this initial value problem define, for each fixed $t$, a hydrodynamically normalized conformal map $g_t \colon \Hy \backslash K_t \rightarrow \Hy$, where $K_t$ is a half-plane hull \cite[Section 4.1]{Law05}. It is easy to see that $K_0 = \emptyset$, as $g_0 = \id_{\Hy}$. Moreover, the hulls form an increasing family, in the sense that $K_s \subsetneq K_t$ whenever $0 \leq s < t \leq T$. This implies that $\hcap(K_t)$ is increasing in $t$. Indeed, more is true: $\hcap(K_t) = 2t$ for each $t$. The fact that half-plane capacity increases continuously means, geometrically, that the hulls are ``growing continuously." We see, then, that the family of hulls produced by a continuous driving term is an example of a \ti{geometric Loewner chain}, according to the following definition.

\begin{definition}
A one-parameter family of half-plane hulls $K_t$, for $0 \leq t \leq T$, is called a \ti{(geometric) Loewner chain} if
\begin{enumerate}
\item[\textup{(i)}] $K_0 = \emptyset$,
\item[\textup{(ii)}] $K_s \subsetneq K_t$ for each $0 \leq s <t \leq T$, and
\item[\textup{(iii)}] $\hcap(K_t) = 2t$ for each $t$.
\end{enumerate}
\end{definition}

It is not true, however, that \ti{every} geometric Loewner chain corresponds to a continuous driving function. In fact, having a driving function places strong restrictions on how the hulls grow infinitesimally. A good example of Loewner chains that do have a continuous driving term are those that are \ti{generated by a curve}, i.e., there is a path $\gamma \colon [0,T] \rightarrow \cl{\Hy}$ for which $\Hy \backslash K_t$ is the unbounded component of $\Hy \backslash \gamma[0,t]$ for each $t$. In the case that $\gamma$ is a simple curve, this chain has an important property: whenever $0<t-s \ll 1$ is very small, the set $K_t \backslash K_s$ has small diameter. In particular, $K_t \backslash K_s$ can be separated from $\infty$ by a cross-cut in $\Hy \backslash K_s$ with diameter going to 0 as $|s-t| \rightarrow 0$. The following theorem, originally proved by C. Pommerenke for Loewner's equation in the unit disk, says that this property characterizes the Loewner chains that correspond to a continuous driving function. See \cite[Theorem 2.6]{LSW01} for a proof of what we state here.

\begin{theorem}[Pommerenke \cite{Pom66}] \label{Pom}
Let $K_t$ be a geometric Loewner chain for $0 \leq t \leq T$. The following are equivalent.
\begin{enumerate}
\item[\textup{(i)}] The hydrodynamic conformal maps $g_t \colon \Hy \backslash K_t \rightarrow \Hy$ satisfy \eqref{LE} for some continuous function $\lambda_t$.
\item[\textup{(ii)}] For each $0\leq s < t \leq T$, the set $K_t \backslash K_s$ can be separated from $\infty$ by a cross-cut in $\Hy \backslash K_s$ of diameter $\leq \omega_1(t-s)$, where $\omega_1(\delta) \rightarrow 0$ as $\delta \rightarrow 0$.
\item[\textup{(iii)}] For each $0\leq s < t \leq T$, the transition hull $K_{s,t} = g_s(K_t\backslash K_s)$ has diameter $\leq \omega_2(t-s)$, where $\omega_2(\delta) \rightarrow 0$ as $\delta \rightarrow 0$.
\end{enumerate}
\end{theorem}

We should remark that condition (iii) does not explicitly appear in Pommerenke's statement or in the statement of \cite[Theorem 2.6]{LSW01}. However, the proof of the equivalence of (i) and (ii) essentially goes through condition (iii). Moreover, it is easy to see that (ii) and (iii) are equivalent, using, for example, Wolff's lemma about length distortion under conformal maps (cf. \cite[Proposition 2.2]{Pom92}).

Pommerenke's theorem sets up a correspondence between continuous driving functions and Loewner chains with harmonically small ``tails" $K_t \backslash K_s$. Let us turn to a few of the finer properties of this correspondence by imposing stronger geometric properties on the hulls or stronger analytic properties on the driving terms.

The first results we should mention in this context concern classical smoothness. Suppose that $K_t$ is a geometric Loewner chain generated by a simple curve, so that $K_T$ is a simple curve itself. In \cite{EE01}, C. Earle and A. Epstein proved that if $K_T$ has a $C^n$ parameterization, for $n \geq 2$, then the associated driving function is $C^{n-1}$. Conversely, if $\lambda_t$ is in the class $C^{\beta}$, with $\beta > 1/2$ and $\beta + 1/2 \notin \N$, then the corresponding simple curve, parameterized by half-plane capacity, is in $C^{\beta + 1/2}$. The case $\beta \leq 2$ was demonstrated by C. Wong \cite{Won14}, and the result for $\beta > 2$ was recently shown by J. Lind and H. Tran \cite{LT14}.

Closer to our interests in this paper, there have been several results dealing with strictly fractional smoothness of $\lambda_t$. Due to certain scaling symmetries in the chordal Loewner equation, it is natural to study the H\"older class $\Lip(1/2)$. Namely, let
$$|| \lambda ||_{1/2} = \sup_{0 \leq s\neq t \leq T} \frac{| \lambda_s - \lambda_t|}{\sqrt{|s-t|}}$$
be the $\Lip(1/2)$ semi-norm. We say that $\lambda$ is in $\Lip(1/2)$ if $|| \lambda ||_{1/2} < \infty$. The following theorem gives a relationship between $\Lip(1/2)$ driving terms and Loewner chains that are generated by quasi-arcs.

\begin{theorem}[Marshall--Rohde \cite{MR05}, Lind \cite{Lin05}, Rohde--Tran--Zinsmeister \cite{RTZ13}] \label{MRL}
If $K_t$ is a geometric Loewner chain generated by a quasi-arc that meets $\R$ non-tangentially, then the corresponding driving term is in $\Lip(1/2)$. Conversely, if $\| \lambda \|_{1/2} < 4$, then the corresponding Loewner chain is generated by a quasi-arc that meets $\R$ non-tangentially.
\end{theorem}

In particular, this theorem ensures that if $\| \lambda \|_{1/2} < 4$, then the corresponding Loewner chain is generated by a simple curve. Finding weaker conditions on $\lambda_t$ that still give a simple curve is fairly wide open and seems surprisingly difficult. For a statement about chains generated by curves that are Lipschitz graphs, see \cite[Theorem 1.2]{RTZ13}.

Coming from the opposite direction, J. Lind and S. Rohde have studied driving functions that generate a space-filling curve, finding a second ``transition value" for the $\Lip(1/2)$ norm. Namely, if $\lambda$ is a $\Lip(1/2)$ function that generates a curve with non-empty interior, then $\| \lambda \|_{1/2} \geq 4.0001$. The constant 4.0001 here is most likely not optimal. In the same paper, Lind and Rohde gave interesting (though technical) geometric criteria for a Loewner chain to have a $\Lip(1/2)$ driving term. We state it here, as it motivates one of the results in this paper.

\begin{theorem}[Lind--Rohde \cite{LR12}] \label{LR1/2}
Let $K_t$ be a geometric Loewner chain, generated by a continuous driving term $\lambda_t$, for $0 \leq t \leq T$. Suppose that there are constants $C_0 > 0$ and $k < \infty$ such that for each $0\leq s < T$, there is a $k$-quasi-disk $D_s \subset \Hy \backslash K_s$, with $\infty \in \cl{D_s}$, for which
\begin{enumerate}
\item[\textup{(i)}] $K_T \backslash K_s \subset \cl{D_s}$ and
\item[\textup{(ii)}] $\diam(K_t \backslash K_s) \leq C_0 \max\{\dist(z,\partial D_s) : z \in K_t \backslash K_s \}$ for all $s < t \leq T$.
\end{enumerate}
Then $\| \lambda \|_{1/2} \leq C$, where $C < \infty$ is a constant depending only on $C_0$ and $k$.
\end{theorem}

The $\Lip(1/2)$ theme in the study of Loewner's equation has been motivated primarily by properties of the $\SLE$ processes we mentioned earlier. Indeed, the class $\Lip(1/2)$ allows for driving functions that generate non-smooth hulls and non-smooth curves. To discuss $\SLE$ more thoroughly, fix $\kappa > 0$ and let $\lambda_t = \sqrt{\kappa}B_t$, where $B_t$ is a one-dimensional Brownian motion. Almost surely, there is a corresponding geometric Loewner chain (cf. Pommerenke's theorem); moreover, this chain is generated by a curve, $\gamma$. It turns out that these random $\SLE_{\kappa}$ curves exhibit two phase transitions, at $\kappa =4$ and $\kappa = 8$ \cite{LSW04, RS05}:
\begin{enumerate}
\item[(i)] if $0 \leq \kappa \leq 4$, then almost surely $\gamma$ is simple and intersects $\R$ in only one point;
\item[(ii)] if $4 < \kappa < 8$, then almost surely $\gamma$ intersects itself but is not space-filling; and 
\item[(iii)] if $\kappa \geq 8$, then almost surely $\gamma$ is space-filling.
\end{enumerate}
Note that this behavior mirrors the transition values we discussed above for the $\Lip(1/2)$ norm.

The relationship between $\SLE$ curves and Loewner chains that are generated by $\Lip(1/2)$ functions is only motivational, though. Indeed, Brownian motion is, almost surely, \ti{not $\Lip(1/2)$}, and the corresponding $\SLE$ curves are, almost surely, \ti{not quasi-arcs}. It is true, however, that Brownian motion is \ti{almost} $\Lip(1/2)$:
\begin{equation} \label{Levy}
|B_s - B_t| \leq \sqrt{2|s-t| \log \lp 1/|s-t| \rp}
\end{equation}
for $|s-t|$ small enough. More generally, a function $\lambda$ is said to be in the weak-$\Lip(1/2)$ class if
$$|\lambda_s - \lambda_t | \leq \sqrt{|s-t|} \cdot \phi \lp 1/|s-t| \rp,$$
where $\phi(x) = o_{\e}(x^{\e})$ as $x \rightarrow \infty$ for all $\e > 0$. The bound in \eqref{Levy} shows that Brownian motion is weak-$\Lip(1/2)$ almost surely.

In addition to weak-type regularity for $\SLE$ driving terms, one can also say something nice about the geometry of $\SLE$ domains. For this we need a definition.

\begin{definition}
Let $\Omega = \Hy \backslash K$, where $K$ is a half-plane hull, and let $f \colon \Hy \rightarrow \Omega$ be the hydrodynamic conformal map. We say that $\Omega$ is a \ti{$(\beta,C_0)$-H\"older domain}, with $0 < \beta \leq 1$ and $C_0 > 0$ if
\begin{equation} \label{Holderf}
|f(z)-f(z')| \leq C_0 \max \{ |z-z'|^{\beta}, |z-z'| \}
\end{equation}
for all $z,z' \in \Hy$. We say that $\Omega$ is a \ti{H\"older domain} if it is a $(\beta,C_0)$-H\"older domain for some choice of $\beta$ and $C_0$.
\end{definition}

Note that if $\Omega$ is a H\"older domain, the hydrodynamic map $f \colon \Hy \rightarrow \Omega$ extends continuously to $\R$, thereby giving $\partial \Omega$ a H\"older continuous parameterization. We will discuss other geometric properties of H\"older domains in the next section. For now, we record the following theorem, which we mentioned informally earlier.

\begin{theorem}[Rohde--Schramm {\cite[Theorem 5.2]{RS05}}]
Let $K_t$ be an $\SLE_{\kappa}$ chain, for $\kappa \neq 4$. Then almost surely, the domains $\Omega_t = \Hy \backslash K_t$ are H\"older domains, with constants depending on $\kappa$, $t$, and on the randomness.
\end{theorem}

The main goal of this paper is to extend the themes taken up by Marshall, Rohde, Lind, Tran, and Zinsmeister to a context that deals with weak-$\Lip(1/2)$ driving terms and H\"older geometry of the corresponding domains. Ideally, one would like to find weak-type regularity assumptions on $\lambda$ that ensure that the associated domains are H\"older. This seems to be a difficult question, so we focus instead on the converse. In our first theorem, we show that H\"older geometry in the domains indeed gives weak-$\Lip(1/2)$ control on the driving term, at least in the case that the Loewner chain is generated by a simple curve.

\begin{theorem} \label{slit}
Let $K_t$ be a geometric Loewner chain, for $0 \leq t \leq T$, that is generated by a simple curve. If $\Hy \backslash K_T$ is a $(\beta,C_0)$-H\"older domain, then this chain has a driving term, $\lambda$, for which 
$$|\lambda_s - \lambda_t| \leq C\sqrt{|s-t| \log(1/|s-t|)}.$$
if $|s-t| \leq 1/2$. Here, $C > 0$ depends only on $\beta$, $C_0$, and $T$.
\end{theorem}

In some sense, this result can be viewed as a deterministic ``converse" to what we know about $\SLE_{\kappa}$ curves when $\kappa < 4$. For such values of $\kappa$, the driving term is a scaled Brownian motion and the corresponding hulls are generated by a simple curve whose complement is a H\"older domain. 

We must ask, then: what can be said in the absence of the ``simple curve" hypothesis? Here, statements are not as neat as in Theorem \ref{slit}, but we still can say something interesting. The conditions we will impose on the Loewner chains are, however, fairly technical, and so we put it off until the end of Section \ref{prelim}. They arise quite naturally, though, as a generalization of the following geometric criteria that guarantee the existence of a driving term in $\Lip(1/2)$.

\begin{theorem} \label{Johnprop}
Let $K_t$ be a geometric Loewner chain, and let $\Omega_t = \Hy \backslash K_t$ be the complementary domains. Suppose that, for each $0 \leq s < t \leq T$, there is a point $z_0 \in K_t \backslash K_s$ for which
\begin{enumerate}
\item[\textup{(i)}] $\diam_{\Omega_s}(K_t \backslash K_s) \leq C_0 \dist(z_0,\partial \Omega_s),$ and
\item[\textup{(ii)}] there is an $L$-John curve in $\Omega_s$ with tip $z_0$ and base-point $\infty$.
\end{enumerate}
Then this chain has a continuous driving term, $\lambda$, for which $\| \lambda \|_{1/2} \leq C$, where $C > 0$ depends only on $L$ and $C_0$.
\end{theorem}

This result extends the criteria given in Theorem \ref{LR1/2} above. Indeed, any point in the $k$-quasi-disk $D_s$ can be connected to $\infty$ by an $L$-John curve in $D_s$, where $L$ depends only on $k$.

In the next section, we will introduce additional terminology and establish some basic estimates that will be useful later on. There we will also state Theorem \ref{nonslit}, which gives criteria to ensure that a Loewner chain be generated by a weak-$\Lip(1/2)$ driving term. Section \ref{secslit} is devoted to the proof of Theorem \ref{slit} and so deals primarily with half-planes slit by a simple curve. In Section \ref{secnonslit} we will prove Theorems \ref{Johnprop} and \ref{nonslit} after discussing their relationship to Theorem \ref{slit}.

\subsection*{Acknowledgements}
The author thanks Mario Bonk for his detailed explanations of the background theory for the Loewner equation and of related topics in conformal mapping. He is grateful to Steffen Rohde for several discussions and for sharing his insight into $\SLE$. He also thanks Michel Zinsmeister, Huy Tran, Joan Lind, and Marie Snipes for helpful conversations during various stages of this project.

Part of the research for this project was done during the ``Interactions between Analysis and Geometry" program at IPAM and also during the ``Random Walks and Asymptotic Geometry of Groups" program at IHP. He thanks those institutes for their hospitality and the participants for a stimulating atmosphere. The author also acknowledges partial support from NSF grants DMS-1162471 and DMS-1344959.

\section{Preliminaries} \label{prelim}

In addition to the Euclidean structure that $\Hy$ possesses as a subset of $\C$, it can also be endowed with the Riemannian metric whose length element is 
$$ds = \frac{\sqrt{dx^2+dy^2}}{y} = \frac{|dz|}{y},$$
where $x = \Re(z)$ and $y= \Im(z) >0$. This makes $\Hy$ a model for 2-dimensional hyperbolic space. We will frequently look to the interplay between the hyperbolic and Euclidean structures on $\Hy$.

Let $\Omega \subsetneq \C$ be a simply connected domain; recall that a domain is open and connected. By the Riemann mapping theorem, there is a conformal map $f \colon \Hy \rightarrow \Omega$. As such, the hyperbolic metric on $\Hy$ can be transferred to a hyperbolic metric on $\Omega$, and we will denote it by $\rho_{\Omega}$. More specifically, we let $\rho_{\Omega}(z) |dz|$ be the hyperbolic length element at $z \in \Omega$, where 
$$\rho_{\Omega}(z) = \frac{1}{|f'(w)| \Im(w)}, \hspace{0.3cm} f(w)=z,$$
so that the metric is defined internally by
$$\rho_{\Omega}(z_0,z_1) = \inf_{\gamma} \int_{\gamma} \rho_{\Omega}(z) |dz|,
\hspace{0.3cm} z_0,z_1 \in \Omega$$
with the infimum taken over all rectifiable paths $\gamma$ in $\Omega$ that join $z_0$ and $z_1$. Let us point out a few features of this metric, which follow immediately from the corresponding properties of the hyperbolic metric on $\Hy$. 

First, any two points in $\Omega$ can be joined by a unique geodesic segment, and the metric is geodesically complete. Thus, $\Omega$ has many hyperbolic lines (isometric images of $\R$) and hyperbolic rays (isometric images of $[0,\infty)$). Second, if $\ell \subset \Omega$ is a hyperbolic line and $z_0 \in \Omega$ is a point, then there is a unique point $p \in \ell$ for which
$$\rho_{\Omega}(z_0,p) = \min \{\rho_{\Omega}(z_0,z) : z \in \ell \}.$$ 
This nearest-point projection, $p$, is characterized by the property that the hyperbolic line joining $z$ and $p$ intersects $\ell$ orthogonally.

We use standard notation regarding the Euclidean geometry of $\C$. Let $B(z_0,r)$ denote the open Euclidean ball of radius $r>0$ in $\C$, centered at $z_0 \in \C$. We will use $\cl{B}(z_0,r)$ to denote the closed Euclidean ball. For $0 \leq r < R$, we let 
$$A(z_0,r,R) := B(z_0, R) \backslash \cl{B}(z_0,r)$$
be the open annulus with inner radius $r$ and outer radius $R$. If $A,B \subset \C$, then 
$$\diam(A) := \sup \{|a-a'| : a,a' \in A\}$$
and
$$\dist(A,B) := \inf \{ |a-b| : a \in A \text{ and } b \in B \}.$$
For a domain $\Omega \subset \C$ and non-empty subsets $A,B \subset \Omega$, let
$$\dist_{\Omega}(A,B) := \inf \left\{ \diam(E) : E\subset \Omega \text{ is connected, }A \cap E \neq \emptyset, \text{ and } B \cap E \neq \emptyset \right\}$$
be the \ti{internal} distance between $A$ and $B$. It is clear that one can require $E$ to be an arc in the infimum. In the case that $A =\{a\}$ and $B=\{b\}$ are singletons, we abuse notation and simply write $\dist_{\Omega}(a,b)$. Similarly, the internal diameter of $A \subset \Omega$ is
$$\diam_{\Omega}(A) := \sup \{ \dist_{\Omega}(a,a') : a,a' \in A \}.$$

Lastly, it will be convenient for us to suppress multiplicative constants in much of our analysis. For quantities $A$ and $B$ that depend on certain choices, the notation $A \lesssim B$ means that there is a constant $C>0$ such that $A \leq C \cdot B$ for all possible choices determining $A$ and $B$. Similarly, the notation $A \approx B$ means that $A \lesssim B$ and $B \lesssim A$. In some instances, the implicit constant $C$ will be an absolute constant, though we will indicate this explicitly for clarity. We should remark that some of the bounds we use can be made sharp without much difficulty, but the extra work does not really improve our results. In fact, the methods we use are probably unable to obtain sharp statements (for example, in Theorem \ref{slit}) due to their inherent ``coarseness." 

\subsection{The quasi-hyperbolic metric}

Let $\Omega \subsetneq \C$ be a simply connected domain. To understand $\rho_{\Omega}$ in terms of the Euclidean geometry of $\Omega$, or of $\partial \Omega$, a useful tool is the comparable quasi-hyperbolic metric. For $z \in \Omega$, let
$$\delta_{\Omega}(z) = \dist(z,\partial \Omega)$$
denote the Euclidean distance from $z$ to $\partial \Omega$. The Koebe distortion theorem guarantees that
\begin{equation} \label{qhcomp}
\frac{1}{2}\rho_{\Omega}(z) \leq \frac{1}{\delta_{\Omega}(z)} \leq 2 \rho_{\Omega}(z),
\end{equation}
cf. \cite[Corollary 1.4]{Pom92}, along with the fact that $\rho_{\D}(z) = 2/(1-|z|^2)$. The quasi-hyperbolic distance is then defined by
$$k_{\Omega}(z_0,z_1) = \inf_{\gamma} \int_{\gamma} \frac{|dz|}{\delta_{\Omega}(z)},
\hspace{0.5cm} z_0,z_1 \in \Omega$$
where, once again, the infimum is taken over all rectifiable paths $\gamma$ in $\Omega$ that join $z_0$ to $z_1$. The estimate in \eqref{qhcomp} immediately implies that
$$\frac{1}{2} \rho_{\Omega}(z_0,z_1) \leq k_{\Omega}(z_0,z_1) \leq 2 \rho_{\Omega}(z_0,z_1)$$
for all $z_0,z_1 \in \Omega$.

To analyze $k_{\Omega}$, it is often helpful to approximate it combinatorially, using Whitney decompositions. We say that a collection $\mathcal{W}$ of closed dyadic Euclidean squares (i.e., side lengths are powers of 2) is a \ti{Whitney decomposition} of $\Omega$ if $\Omega= \bigcup_{Q \in \mathcal{W}} Q$, the squares have non-empty, pairwise disjoint interiors, and 
\begin{equation} \label{Wsquare}
\tfrac{1}{2}\diam(Q) \leq \dist(Q, \partial \Omega) \leq 4\diam(Q)
\end{equation}
for all $Q \in \mathcal{W}$. Note that $\diam(Q) = \sqrt{2}\cdot 2^j$ where $2^j$ is the side length of $Q$. It is not difficult to construct such a decomposition for a given domain $\Omega$; see, for example \cite[Section 6.1]{Ste70}. The precise construction of $\mathcal{W}$ is not, however, important in our analysis. Instead, the key property is that $\diam(Q)$ is comparable to $\dist(Q,\partial \Omega)$ for each $Q \in \mathcal{W}$. For convenience, though, when $\Omega = \Hy$, we will always use the decomposition with squares of the form
$$Q_{j,k} = \{z: k2^j \leq \Re(z) \leq (k+1)2^j \text{ and } 2^j \leq \Im(z) \leq 2^{j+1} \},$$
where $k,j \in \Z$. Let us also remark that, in general, the squares in a Whitney decomposition have controlled overlap, in the sense that no point is contained in more than four squares. This continues to hold if the squares are replaced by slightly larger dilates of themselves.

If $Q \in \mathcal{W}$ is a Whitney square in a decomposition of $\Omega$, then \eqref{Wsquare} easily implies that the diameter of $Q$, with respect to the quasi-hyperbolic metric $k_{\Omega}$, satisfies
$$1/4 \leq \diam_{k_{\Omega}}(Q) \leq 2.$$
In light of this, it is not surprising that one can approximate the quasi-hyperbolic metric (and, therefore, also the hyperbolic metric) in $\Omega$ using chains of Whitney squares, as the following lemma indicates.

\begin{lemma} \label{qhmetric}
Let $\mathcal{W}$ be a Whitney decomposition of $\Omega$, and let $\ell$ be a quasi-hyperbolic geodesic segment in $\Omega$ with endpoints $z_0$ and $z_1$. Define $N(\ell)$ to be the number of squares in $\mathcal{W}$ that $\ell$ intersects. Then 
$$N(\ell) \lesssim \max \{k_{\Omega}(z_0,z_1), 1\} \leq 2N(\ell),$$
where the implicit constant is absolute.
\end{lemma}

Note that the upper bound follows immediately from $\diam_{k_{\Omega}}(Q) \leq 2$ for each $Q \in \mathcal{W}$. The lower bound is not difficult to establish using the fact that each Whitney square $Q$ that intersects $\ell$ has a neighbor through which $\ell$ runs for length comparable to $\diam(Q)$.

\subsection{Geometry of half-plane hulls}

The standard Whitney decomposition of $\Hy$ is also important for a geometric understanding of the half-plane capacity. Recall that for a half-plane hull $K \subset \Hy$, i.e., $K$ is bounded and $\Hy \backslash K$ is a simply connected domain, $\hcap(K)$ is the non-negative number for which the hydrodynamic conformal map $f \colon \Hy \rightarrow \Hy \backslash K$ has
$$f(z) = z - \frac{\hcap(K)}{z} + O \lp \frac{1}{z^2} \rp, \hspace{0.3cm} \text{as } z \rightarrow \infty.$$
The following lemma tells us that $\hcap(K)$ is comparable to the total area of Whitney squares that $K$ intersects.

\begin{lemma} \label{Warea}
Let $K$ be a half-plane hull, and let $\mathcal{W}$ be the standard Whitney decomposition of $\Hy$. Define $\Area_{\mathcal{W}}(K)$ to be the area of the union of squares $Q \in \mathcal{W}$ for which $Q \cap K \neq \emptyset$. Then $\hcap(K) \approx \Area_{\mathcal{W}}(K)$, where the implicit constants are absolute.
\end{lemma}

There are multiple proofs of this lemma; see, for example \cite{LLN09} and \cite{RW14}. In these references, it is shown, respectively, that $\hcap(K)$ is comparable to the area of the union of disks tangent to $\R$ with centers in $K$ and to the Euclidean area of the hyperbolic $1$-neighborhood of $K$ in $\Hy$. These two area quantities are, of course, comparable to each other and also comparable to the Whitney area $\Area_{\mathcal{W}}(K)$. See also \cite[Chapter 3, Section 4]{Law05} for an interpretation of half-plane capacity in terms of hitting probabilities of Brownian motion started near infinity.

Analysis of the hydrodynamic map $f \colon \Hy \rightarrow \Hy \backslash K$ is made easier by the integral representation of $f$. Namely, there is a finite, compactly supported measure $\mu_K$ on $\R$ for which
$$f(z) = z - \int_{\R} \frac{d\mu_K(u)}{z-u}, \hspace{0.5cm} z \in \Hy.$$
See, for example, \cite[Proposition 2.1]{Bau05}. If $f$ extends continuously to $\R$, then $d\mu_K(u) = \tfrac{1}{\pi}\Im(f(u))du$. In any case, a geometric series expansion of the integral representation shows immediately that $\mu_K(\R) = \hcap(K)$. 

The support of $\mu_K$ is defined to be
$$\supp(\mu_K) = \{ x \in \R : \mu_K((x-r,x+r)) > 0 \text{ for all } r > 0 \},$$
which is a compact set in $\R$. It is an important fact that if $\supp(\mu_K) \subset (x-r,x+r)$, then $f$ can be extended, via Schwarz reflection, to a \ti{univalent} map on $\C \backslash \cl{B}(x,r)$. The smallest possible value for $r$ here is, of course, $\diam(\supp(\mu_K))/2$. Similarly, if $K \subset B(x,r)$ for some point $x \in \R$, then $f^{-1}$ can be extended to a univalent map on $\C \backslash \cl{B}(x,r)$. The smallest value for $r$ in this situation is the \ti{radius of $K$}:
$$\rad(K) = \inf \{ r >0 : \text{there is } x \in \R \text{ with } K \subset B(x,r) \}.$$
Obviously, we have $\rad(K) \leq \diam(K) \leq 2\rad(K)$, but it will be more convenient to use the radius in the following distortion estimates.

\begin{lemma} \label{supp}
Let $K$ be a half-plane hull and let $f \colon \Hy \rightarrow \Hy \backslash K$ be its hydrodynamic conformal map.
\begin{enumerate}
\item[\textup{(i)}] If $\supp(\mu_K) \subset (x-r,x+r)$ then $\Hy \backslash B(x,2r) \subset f(\Hy \backslash B(x,r))$. In particular, $K \subset B(x,2r)$, so that $\rad(K) \leq \diam(\supp(\mu_K))$.
\item[\textup{(ii)}] If $K \subset B(x,r)$, for some $x \in \R$, then $\Hy \backslash B(x,2r) \subset f^{-1}(\Hy \backslash B(x,r))$. In particular, $\supp(\mu_K) \subset (x-2r,x+2r)$, so that $\diam(\supp(\mu_K)) \leq 4 \rad(K)$.
\item[\textup{(iii)}] If $w \in \supp(\mu_K)$ and $z \in K$, then $|w-z| \leq 4\rad(K)$
\item[\textup{(iv)}] $\hcap(K) \leq (\diam(\supp(\mu_K))/2)^2 \leq 4\rad(K)^2$.
\item[\textup{(v)}] If $z \in \Hy$, then $|f(z)-z| \leq 3\rad(K)$.
\end{enumerate}
\end{lemma}

\begin{proof}
Parts (i) and (ii) follow from the well-known theorem on omitted values for univalent functions. Namely, if $h \colon \C \backslash \cl{\D} \rightarrow \C$ is a univalent map with $h(z) = z + b_1/z + O(1/z^2)$ as $z \rightarrow \infty$, then 
$$\C \backslash \cl{B}(0,2) \subset h(\C \backslash \cl{\D}).$$
In other words, the ``omitted set" is contained in $\cl{B}(0,2)$, cf. \cite[p. 8]{Pom92}. The desired estimates are obtained by using the functions 
$$h(z) = \frac{f(rz + x) -x}{r} \hspace{0.3cm} \text{and} \hspace{0.3cm} h(z) = \frac{f^{-1}(rz + x) -x}{r},$$
respectively, for $|z| > 1$, where $f$ and $f^{-1}$ have been appropriately extended via Schwarz reflection.

Part (iii) follows immediately from part (ii). Indeed, if $K \subset B(x,r)$ then we have $K \cup \supp(\mu_K) \subset B(x,2r)$ so that $|w-z| \leq 4r$.

Part (iv) is a consequence of the area theorem: for a univalent map $h$ as before, $|b_1| \leq 1$ \cite[Theorem 2.1]{Dur83}. If $\supp(\mu_K) \subset (x-r,x+r)$, then we can apply this estimate to the function
$$h(z) = \frac{f(rz + x) - x}{r} = z - \frac{\hcap(K)/r^2}{z} + O \lp \frac{1}{z^2} \rp,$$
defined for $|z| > 1$ using a suitable Schwarz reflection. We deduce then that $\hcap(K) \leq r^2$.

Lastly, part (v) is essentially Corollary 3.44 in \cite{Law05}. There, $\rad(K)$ is defined to be the smallest radius of a disk centered at 0 that contains $K$. The proof, however, is the same: one can simply translate $K$ so that $K \subset \cl{B}(0,\rad(K))$.
\end{proof}

We will use these estimates frequently in later arguments. Overall, sharp constants are not essential, though it will be necessary to have \ti{specified} constants. For now, one should focus on the fact that $\supp(\mu_K)$ and $K$ have comparable diameters, that $\hcap(K)$ is bounded from above by a multiple of $\diam(K)^2$, and that the map $f$ does not move points by more than distance $3\diam(K)$.

Later, we will need some more precise distortion estimates for the hydrodynamic map $f \colon \Hy \rightarrow \Hy \backslash K$. First, let us remark on the behavior of points that are far from $K$. If $z \in \Hy$ has $\dist(z,K) \geq 7\diam(K)$, then part (iii) of the previous lemma gives
$$\dist(z,\supp(\mu_K)) \geq 7\diam(K) - 4\diam(K) = 3\diam(K) \geq 3\rad(K).$$
The integral representation of $f$ then shows that
\begin{equation} \label{fardist}
|f'(z) - 1| \leq \int_{\R} \frac{d\mu_K(u)}{|z-u|^2} \leq \frac{\mu_K(\R)}{9\rad(K)^2} = \frac{\hcap(K)}{9\rad(K)^2} \leq \frac{1}{2},
\end{equation}
where the last inequality comes from Lemma \ref{supp}(iv). Thus, $f$ is bi-Lipschitz in this region. For other points in $\Hy$, we have the following lemma.

\begin{lemma} \label{neardist}
Let $K$ be a half-plane hull and $f \colon \Hy \rightarrow \Hy \backslash K$ the associated hydrodynamic map. If $\Im(z) \leq 10\diam(K)$, then
$$\Im(z) \lesssim \diam(K) \cdot |f'(z)|, $$
with an absolute implicit constant. In particular, $\Im(z)^2 \lesssim \diam(K) \cdot \delta_{\Hy \backslash K}(f(z))$,
again with an absolute constant.
\end{lemma}

\begin{proof}
By the Koebe 1/4-theorem \cite[Theorem 2.3]{Dur83}, we know that 
$$|f'(z)|\Im(z) \leq 4 \delta_{\Hy \backslash K}(f(z)).$$
It therefore suffices to prove the first statement.

To this end, let $z_0 \in \Hy$ be the point with $\Re(z_0) = \Re(z)$ and $\Im(z_0) = 10\diam(K)$. Note that $z \in B(z_0,\Im(z_0)) \subset \Hy$ due to the assumption that $\Im(z) \leq 10\diam(K)$. The Koebe distortion theorem \cite[Theorem 2.5]{Dur83}, applied to this ball $B(z_0,\Im(z_0))$, gives
$$|f'(z)| \geq \frac{1-r}{(1+r)^3} |f'(z_0)| \geq \frac{1-r}{8} |f'(z_0)|,$$
where $r = |z-z_0| / \Im(z_0) = (\Im(z_0) - \Im(z)) / \Im(z_0)$. Thus, we obtain
$$|f'(z)| \geq \frac{\Im(z)}{8\Im(z_0)}|f'(z_0)| \gtrsim \frac{\Im(z)}{\diam(K)}|f'(z_0)|.$$
The desired conclusion then follows from \eqref{fardist}, which shows that $|f'(z_0)| \geq 1/2$.
\end{proof}

\subsection{John domains and H\"older domains}

In this paper, we are primarily interested in the class of H\"older domains and their corresponding geometric properties in relation to the Loewner equation. The motivation for two of our theorems, however, comes from the geometry associated to a smaller class of domains, namely, the John domains. We think it is appropriate to mention these briefly. Moreover, the curves that are used to define the John condition appear in the statements of these two theorems.

Let $\Omega \subset \C$ be a domain. For the most part, we will focus on the case where $\Omega = \Hy \backslash K$, but for now this is not important. Let $\alpha \subset \Omega$ be a simple curve in $\Omega$. For two points $x,y \in \alpha$ lying on the curve, we use $\alpha[x,y]$ to denote the closed sub-arc that joins $x$ and $y$. 

\begin{definition}
A simple curve $\alpha \subset \Omega$ is called an \ti{$L$-John curve}, with $L \geq 1$, if it has an endpoint $z \in \Omega$ such that $\diam(\alpha[x,z]) \leq L \delta_{\Omega}(x)$ for each $x \in \alpha$. When this holds, we say that the point $z$ is the \ti{tip} of $\alpha$, and we call the other endpoint, $z_0 \in \Omega \cup \{\infty\}$, the \ti{base-point} of $\alpha$.
\end{definition}

Note that we require the tip of any John curve to lie in the domain $\Omega$, but we allow the base-point to be at infinity. Here, having an endpoint at infinity simply means that the curve accumulates to infinity. It is easy to see that the John condition prohibits the curve from accumulating to any point in $\partial \Omega \backslash \{\infty\}$. For the most part, the John curves we use will lie in domains of the form $\Hy \backslash K$ and will in fact have base-points at infinity.

Geometrically, one should think of a John curve as the core of a ``twisted cone" in $\Omega$ that joins the base-point $z_0$ to the tip $z$. This cone is formed by the union of the balls 
$$B(x,\diam(\alpha[x,z])/L) \subset \Omega,$$
with $x \in \alpha$. For $z \in \Omega$ close to $\partial \Omega$, the existence of a John curve with tip $z$ means that $z$ is nicely accessible if one starts at the base-point $z_0$. The domain $\Omega$ is called an $L$-John domain if there is a point $z_0 \in \Omega \cup \{\infty\}$ such that every $z \in \Omega$ is the tip of some $L$-John curve in $\Omega$ with base-point $z_0$. This condition means that all points of $\Omega$ are nicely accessible from the common base-point $z_0$.

Given the interpretation of John curves in terms of accessibility, it is not surprising that one can bound the quasi-hyperbolic distance between the base-point $z_0$ and arbitrary points on the curve. In fact, we will later need this type of statement for the hyperbolic metric, so we now focus on the case that $\Omega$ is a simply connected domain. The following lemma shows that John curves give bounds on $\rho_{\Omega}$ that mimic the logarithmic-type bounds on $\rho_{\Hy}$.

\begin{lemma} \label{Johncone}
Let $\Omega$ be a simply connected domain, and let $\alpha$ be an $L$-John curve in $\Omega$ with base-point $z_0 \in \Omega \cup \{\infty\}$ and tip $z \in \Omega$. Then
$$\rho_{\Omega}(x,z) \leq \frac{1}{\beta} \log \lp \frac{\delta_{\Omega}(x)}{\delta_{\Omega}(z)} \rp + C$$
for each $x \in \alpha$, where $0< \beta \leq 1$ and $C >0$ depend only on $L$.
\end{lemma}

The proof of this lemma follows standard techniques used to study the quasi-hyperbolic metric. For similar statements, see \cite[Section 3]{GHM89}.

\begin{proof}
Let $\mathcal{W}$ be a Whitney decomposition of $\Omega$, and fix $x \in \alpha$. Recall that $\diam_{k_{\Omega}}(Q) \leq 2$ for each $Q \in \mathcal{W}$, so 
\begin{equation} \label{C_0}
\rho_{\Omega}(z,x) \leq 2k_{\Omega}(z,x) \leq 4 N(\alpha[x,z]),
\end{equation}
where $N(\alpha[x,z])$ is the number of squares in $\mathcal{W}$ that $\alpha[x,z]$ intersects.

If $x$ lies in a Whitney square that intersects a Whitney square containing $z$, then $\rho_{\Omega}(x,z) \leq 8$, and the desired inequality holds by taking $C$ large enough. Otherwise, if $Q \in \mathcal{W}$ contains $z$, then $\diam(\alpha[x,z]) \geq \diam(Q)/15$, and so
$$\delta_{\Omega}(z) \leq \diam(Q) + \dist(Q,\partial \Omega) \leq 5\diam(Q) \leq 75 \diam(\alpha[x,z]).$$
Thus, we may assume that $\delta_{\Omega}(z) \leq 75\diam(\alpha[x,z])$.

Using the $L$-John condition, we first show that any $Q \in \mathcal{W}$ intersecting $\alpha[x,z]$ must have
\begin{equation} \label{Qrest}
\frac{\delta_{\Omega}(z)}{10L} \leq \diam(Q) \leq 200L \delta_{\Omega}(x).
\end{equation}
Indeed, if $w \in Q \cap \alpha[x,z]$, then the Whitney property implies that
$$\delta_{\Omega}(w) \leq \diam(Q) + \dist(Q,\partial \Omega) \leq 5\diam(Q),$$
and so the $L$-John property gives
$$\delta_{\Omega}(z) \leq \diam(\alpha[w,z]) + \delta_{\Omega}(w) \leq (L+1) \delta_{\Omega}(w) \leq 10L\diam(Q).$$
Similarly, we have 
$$\begin{aligned}
\diam(Q) &\leq 2\dist(Q,\partial \Omega) \leq 2\delta_{\Omega}(w) \leq 2(\diam(\alpha[w,z]) + \delta_{\Omega}(z)) \\
&\leq 200\diam(\alpha[x,z]) \leq 200L\delta_{\Omega}(x),
\end{aligned}$$
which gives the other inequality in \eqref{Qrest}.

For each $j \in \Z$, let $\mathcal{W}_j$ be the collection of $Q \in \mathcal{W}$ for which $Q \cap \alpha[x,z] \neq \emptyset$ and $\diam(Q) = \sqrt{2}\cdot 2^j$. Then \eqref{Qrest} implies that $\mathcal{W}_j$ is empty unless
$$\log_2 \lp \frac{\delta_{\Omega}(z)}{20L} \rp \leq j \leq \log_2 \lp 200L\delta_{\Omega}(x) \rp.$$
Moreover, if $Q \in \mathcal{W}_j$, then the $L$-John condition ensures that $Q \subset B(z,10L\cdot 2^j)$. Any such $Q$ contains a ball of radius $\geq 2^{j-1}$ in its interior, so the doubling property of Lebesgue measure guarantees that $\# \mathcal{W}_j \leq C_1$, where $C_1$ depends only on $L$. We can then estimate
$$N(\alpha[x,z]) = \sum_{j \in \Z} \# \mathcal{W}_j \leq C_2 \log \lp \frac{\delta_{\Omega}(x)}{\delta_{\Omega}(z)} \rp + C_3,$$
where $C_2,C_3 \geq 1$ depend only on $L$. The desired conclusion now follows from the bound in \eqref{C_0}.
\end{proof}

Recall from the previous section that a domain $\Omega = \Hy \backslash K$, which is the complement in $\Hy$ of a half-plane hull $K$, is said to be a $(\beta,C_0)$-H\"older domain if the hydrodynamic conformal map $f \colon \Hy \rightarrow \Omega$ satisfies the H\"older continuity condition
$$|f(z) - f(z')| \leq C_0 \max\{ |z-z'|^{\beta}, |z-z'| \}$$
for all $z,z' \in \Hy$. We should remark that, in the literature, H\"older domains are typically defined as H\"older continuous conformal images of $\D$ and, as such, are necessarily bounded. Many of the facts that we establish here have classical analogs for bounded H\"older domains, and although it is a technical nuisance to work with unbounded domains, the ideas we use are very similar to those found in the bounded setting. We begin with the following lemma which gives an equivalent condition for the H\"older property in terms of $|f'|$ (see \cite[Theorem 5.1]{Dur70} for the analogous statement about bounded H\"older domains). 

\begin{lemma} \label{f'equiv}
Let $f \colon \Hy \rightarrow \Omega$ be a conformal map onto a domain $\Omega$, and let $0 < \beta \leq 1$. Then the following are quantitatively equivalent.
\begin{enumerate}
\item[\textup{(i)}] There is $C_0 > 0$ such that $|f(z)-f(z')| \leq C_0 \max\{ |z-z'|^{\beta}, |z-z'| \}$ for all $z,z' \in \Hy$.
\item[\textup{(ii)}] There is $C_1 > 0$ such that $|f'(z)| \leq C_1 \max \{\Im(z)^{\beta - 1}, 1\}$ for all $z \in \Hy$.
\end{enumerate}
By quantitatively equivalent, we mean that $C_1$ depends only on $\beta$ and $C_0$, and that $C_0$ depends only on $\beta$ and $C_1$.
\end{lemma}

\begin{proof}
The implication ``(i) implies (ii)" follows almost immediately from the Koebe distortion theorem. Indeed, for $z \in \Hy$, let $B := B(z,\Im(z)/2)$. Then we have $|f'(z)|\diam(B) \leq 10\diam(f(B))$. Using the bounds in (i), we also know that $\diam(f(B)) \leq C_0 \max \{\diam(B)^{\beta}, \diam(B) \}$, so
$$|f'(z)| \lesssim \max\{ \Im(z)^{\beta}, \Im(z) \} / \Im(z) \lesssim \max \{ \Im(z)^{\beta -1}, 1 \},$$
where the implicit constant depends only on $C_0$ and $\beta$.

The opposite implication is slightly more difficult but straightforward. The idea is to bound $|f(z) -f(z')|$ by the length of the hyperbolic geodesic in $f(\Hy)$ that joins the two points. It is easier to carry this out in the following way. Without loss of generality, assume that $\Im(z') \leq \Im(z)$. As a first case, we suppose that $\Im(z) < 2|z-z'|$. Let $\gamma$ and $\gamma'$ denote, respectively, the vertical segments beginning at $z$ and $z'$, going up to height $2|z-z'|$. Let $w$ and $w'$ denote, respectively, their endpoints at height $2|z-z'|$ so that $w' \in B := B(w, \Im(w)/2)$. We then estimate
$$|f(z)-f(z')| \leq \length(f(\gamma)) + \length(f(\gamma')) + \diam(f(B)).$$
Notice that
$$\begin{aligned}
\length(f(\gamma)) &= \int_{\Im(z)}^{2|z-z'|} |f'(\Re(z)+it)| dt \leq C_1 \int_{\Im(z)}^{2|z-z'|} \max \{t^{\beta-1},1\}dt \\
&\lesssim \max \{|z-z'|^{\beta}, |z-z'| \},
\end{aligned}$$
where the implicit constant depends only on $C_1$ and $\beta$. A similar bound holds for $\length(f(\gamma'))$. To control $\diam(f(B))$, we use the Koebe distortion theorem again:
$$\begin{aligned}
\diam(f(B)) &\leq 10|f'(w)| \diam(B) \lesssim \max \{\Im(w)^{\beta-1},1\} \Im(w) \\
&\lesssim \max \{|z-z'|^{\beta}, |z-z'| \},
\end{aligned}$$
with implicit constant depending only on $C_1$ and $\beta$. This gives the desired bound on $|f(z)-f(z')|$.

For the remaining case, when $\Im(z) \geq 2|z-z'|$, we once again use the Koebe distortion theorem, along with the fact that $z' \in \cl{B}(z,\Im(z)/2)$, to obtain
$$\begin{aligned}
|f(z)-f(z')| &\leq 10|f'(z)| \cdot |z-z'| \lesssim \max \{\Im(z)^{\beta-1}, 1\}|z-z'| \\
&\lesssim \max \{|z-z'|^{\beta}, |z-z'| \},
\end{aligned}$$
where the constant again depends only on $C_1$ and $\beta$.
\end{proof}

A second condition that characterizes H\"older domains, originally due to J. Becker and C. Pommerenke \cite{BP82} in the bounded setting, is a growth condition on the hyperbolic metric for points close to the boundary.

\begin{lemma} \label{hypgrowth}
Let $\Omega = \Hy \backslash K$, where $K \neq \emptyset$ is a half-plane hull, and assume that $\Omega$ is a $(\beta, C_0)$-H\"older domain. Then there is a point $z_0 \in \Omega$ and a constant $C>0$, depending only on $\beta$ and $C_0$, such that
\begin{equation} \label{Holderrho}
\rho_{\Omega}(z_0,z) \leq \frac{1}{\beta} \log \lp \frac{\max\{ \diam(K)^{\beta}, \diam(K)\} }{\delta_{\Omega}(z)} \rp + C
\end{equation}
for all $z \in \Omega$ that lie on a hyperbolic geodesic segment joining $z_0$ to a point in the set $D_K := \{w \in \Omega : \dist(w,K) \leq 100 \diam(K) \}$.
\end{lemma}

Before giving the proof, we should make a few remarks. First, the set $D_K$ is not important in itself. Indeed, it is straightforward to show that the growth condition in \eqref{Holderrho} for points $z$ lying on geodesics that join $z_0$ to $D_K$ is equivalent to the same growth condition for all points $z$ in any fixed \ti{bounded} set that contains $K$ in its interior, as long as one is willing to change the additive constant $C$. We prefer, however, to specify $D_K$ in order to obtain quantitative control of constants. In a similar way, one could change the base-point $z_0 \in \Omega$, as long as $C$ was also allowed to change. Lastly, one could subsume the numerator $\max\{\diam(K)^{\beta}, \diam(K)\}$ into the additive constant $C$, but it will be important for us to keep track of how $\diam(K)$ affects the constants.

\begin{proof}
For notational ease, let $r = \diam(K) >0$, and let $x \in \R$ be a point in $\cl{K}$. By Lemma \ref{supp}(v), we know that
\begin{equation} \label{incl}
D_K \subset f(B(x,104r)).
\end{equation}
Let $w_0 = x + 104r i \in \Hy$, and set $z_0 = f(w_0) \in \Omega$. Using Lemma \ref{supp}(v) again, we observe that
$$\dist(z_0,K) \leq 107 r \hspace{0.3cm} \text{and} \hspace{0.3cm} \Im(z_0) \geq 100r.$$

Now, fix $z \in \Omega$ lying on a hyperbolic geodesic segment joining $z_0$ to a point in $D_K$. Let $w = f^{-1}(z) \in \Hy$. From the inclusion in \eqref{incl}, we know that $w$ lies on a hyperbolic geodesic in $\Hy$ that joins $w_0$ to a point in $B(x, 104r)$. In particular, this means that
$$\rho_{\Hy}(w_0,w) \leq \log \lp \frac{\Im(w_0)}{\Im(w)} \rp + C,$$
where $C>0$ is an absolute constant. As $\Omega$ is a $(\beta,C_0)$-H\"older domain, Lemma \ref{f'equiv} gives the estimate $|f'(w)| \lesssim \max \{\Im(w)^{\beta-1}, 1\}$, with implicit constant depending on $\beta$ and $C_0$. Thus, the Koebe distortion theorem gives
$$\delta_{\Omega}(z) \lesssim |f'(w)| \Im(w) \lesssim \max \{\Im(w)^{\beta},\Im(w) \}.$$
If $\Im(w) \geq 1$, then this implies that
$$\rho_{\Omega}(z_0,z) = \rho_{\Hy}(w_0,w) \leq \log \lp \frac{\Im(w_0)}{\Im(w)} \rp + C 
\leq \log \lp \frac{\diam(K)}{\delta_{\Omega}(z)} \rp + C',$$
where $C'>0$ depends only on $\beta$ and $C_0$.
If $\Im(w) < 1$, then we similarly obtain
$$\rho_{\Omega}(z_0,z) \leq \log \lp \frac{\diam(K)}{\delta_{\Omega}(z)^{1/\beta}} \rp + C'' 
\leq \frac{1}{\beta} \log \lp \frac{\diam(K)^{\beta}}{\delta_{\Omega}(z)} \rp + C'',$$
where $C''>0$ again depends only on $\beta$ and $C_0$.
\end{proof}

For later use, we should note that the growth condition in \eqref{Holderrho} implies that $\Omega$ is a $(\beta, C')$-H\"older domain for some $C' >0$. It is important here, however, that $C'$ is allowed to depend on $z_0$, or more precisely, that it is allowed to depend on the distance between $z_0$ and $K$. We record this fact in the following lemma.

\begin{lemma} \label{derivative}
Let $\Omega = \Hy \backslash K$, where $K \neq \emptyset$ is a half-plane hull. Let $z_0 \in \Omega$ be a point for which $\Im(z_0) \geq R$ and $\dist(z_0,K) \leq 100R$ for some value $R \geq 10\diam(K)$. Suppose that there is a constant $C > 0$ such that
$$\rho_{\Omega}(z_0,z) \leq \frac{1}{\beta} \log \lp \frac{1}{\delta_{\Omega}(z)} \rp + C$$
for all $z \in \Omega$ with $\dist(z,K) \leq 10\diam(K)$. Then $\Omega$ is a $(\beta, C')$-H\"older domain, where $C' > 0$ depends only on $\beta$, $C$, and $R$.
\end{lemma}

\begin{proof}
Let $f \colon \Hy \rightarrow \Hy \backslash K$ be the hydrodynamic conformal map. We want to establish the bound 
$$|f'(w)| \lesssim \max \{ \Im(w)^{\beta -1}, 1\}, \hspace{0.3cm} \text{for all } w \in \Hy,$$
where the implicit constants depend only on $\beta$, $C$, and $R$. For notational ease, we again let $r=\diam(K) >0$. First note that if 
$\dist(w,K) \geq 7r$, then the estimate in \eqref{fardist} shows that $|f'(w)| \leq 3/2$. 

Thus, we may assume that $\dist(w,K) < 7r$. In this case, Lemma \ref{supp}(v) gives $\dist(f(w),K) \leq 10r$. Let $z = f(w)$, so our hypothesis implies that
$$\rho_{\Omega}(z_0,z) \leq \frac{1}{\beta} \log \lp \frac{1}{\delta_{\Omega}(z)} \rp + C.$$
Now let $w_0  = f^{-1}(z_0) \in \Hy$. Observe that $\dist(w_0,K) \leq 100R + 3r \leq 101R$ and $\Im(w_0) \geq R - 3r \geq R/2$; in particular, $\dist(w_0,K)$ and $\Im(w_0)$ are both comparable to $R$. This implies that
$$\rho_{\Hy}(w_0,w) \geq \log \lp \frac{1}{\Im(w)} \rp -C',$$
where $C' >0$ depends only on $R$. Together with the upper bound on $\rho_{\Omega}(z_0,z)  = \rho_{\Hy}(w_0,w)$, we have
$$\log \lp \frac{1}{\Im(w)} \rp \leq \frac{1}{\beta} \log \lp \frac{1}{\delta_{\Omega}(z)} \rp + C'',$$
where $C''$ depends only on $C$ and $R$. Thus,
$$\delta_{\Omega}(z) \lesssim \Im(w)^{\beta},$$
with implicit constant depending on $\beta$, $C$, and $R$. To conclude, note that Koebe's 1/4-theorem guarantees that $\delta_{\Omega}(z) \approx |f'(w)| \Im(w)$, with absolute constants, so $|f'(w)| \lesssim \Im(w)^{\beta -1}.$
\end{proof}

Observe that the base-point $z_0$ we found during the proof of Lemma \ref{hypgrowth} had $\dist(z_0,K) \leq 200\diam(K)$ and $\Im(z_0) \geq 100\diam(K)$. In particular, it satisfies the hypotheses in the previous lemma, with $R = 10\diam(K)$.

It turns out that one can strengthen the statement given in Lemma \ref{hypgrowth}, and the stronger version of it will be of particular importance for our analysis of H\"older domains. The result is originally due to W. Smith and D. Stegenga \cite[Theorem 3]{SS90} in the setting of bounded H\"older domains. For completeness, we give a proof, but it is mostly an adaptation of their proof to our setting. Once again, we must keep track of how $\diam(K)$ affects the bounds.

\begin{lemma} \label{hypext}
Assume that $\Omega = \Hy \backslash K$ is a $(\beta,C_0)$-H\"older domain. Let $z_0 \in \Omega$ be the base-point given by Lemma \ref{hypgrowth}. There is a constant $C >0$, depending only on $\beta$ and $C_0$, such that whenever $\ell$ is a hyperbolic geodesic segment from $z_0$ to a point $z_1 \in D_K$, we have
$$\rho_{\Omega}(z_0,x) \leq \frac{1}{\beta} \log \lp \frac{\max \{ \diam(K)^{\beta}, \diam(K) \} }{\length(\ell[x,z_1])} \rp + C$$
for each point $x \in \ell$.
\end{lemma}

\begin{proof}
For simplicity, let $M = \max\{ \diam(K)^{\beta}, \diam(K) \}$. Fix $z_1 \in D_K$, and let $\ell$ be the hyperbolic geodesic segment in $\Omega$ from $z_0$ to $z_1$. Let $C_1$ be a large constant, to be determined, and suppose that there is $x_0 \in \ell$ for which
$$\rho_{\Omega}(z_0,x_0) > \frac{1}{\beta} \log \lp \frac{M}{\length(\ell[x_0,z_1])} \rp + C_1.$$
We wish to arrive at a contradiction if $C_1$ is large enough, depending on $\beta$ and $C_0$.

To this end, let $L = \length(\ell[x_0,z_1])$, and define points $x_k$ recursively by
$$x_k \in \ell[x_{k-1},z_1] \hspace{0.2cm} \text{ such that } \hspace{0.2cm} \length(\ell[x_{k-1},x_k]) = \frac{L}{2^k},$$
for $k \in \N$. Let $\rho_{\Omega}(z)|dz|$ denote the hyperbolic length element in $\Omega$. Recall from earlier that $1/2 \leq \delta_{\Omega}(z) \rho_{\Omega}(z) \leq 2$ for all $z \in \Omega$. Now, we note that if $z \in \ell[x_0,z_1]$, then
$$\begin{aligned}
\frac{1}{\beta} \log \lp \frac{M}{L} \rp + C_1 &< \rho_{\Omega}(z_0,x_0) \leq \rho_{\Omega}(z_0,z)
\leq \frac{1}{\beta} \log \lp \frac{M}{\delta_{\Omega}(z)} \rp + C \\
&\leq \frac{1}{\beta} \log \lp M\rho_{\Omega}(z) \rp + C + \frac{1}{\beta},
\end{aligned}$$
where $C>0$ is the additive constant in Lemma \ref{hypgrowth}. Thus, we find
$$\frac{1}{L\rho_{\Omega}(z)} \leq e^{1+\beta(C-C_1)}$$
for all $z \in \ell[x_0,z_1]$. Choose $C_1$ large enough that $e^{1+\beta(C-C_1)} < \beta/4$, and let
$$\lambda_k = \sup \left\{ \frac{1}{L\rho_{\Omega}(z)} : z \in \ell[x_k,z_1] \right\}$$
for each $k \geq 0$. From what we have already verified, we know that 
$$\lambda_0 \leq e^{1+\beta(C-C_1)} < \beta/4.$$
We claim that, by imposing one more condition on the size of $C_1$, we can ensure that $\lambda_k \leq \lambda_0^{k+1}$ for all $k$. We proceed by induction on $k$, the case $k=0$ being true automatically.

Suppose $\lambda_{k-1} \leq \lambda_0^k$ for some $k \geq 1$. Fix $z \in \ell[x_k,z_1]$ and note that
$$\rho_{\Omega}(x_{k-1},z) = \int_{\ell[x_{k-1},z]} \rho_{\Omega}(w)|dw| \geq \int_{\ell[x_{k-1},z]} \frac{|dw|}{L\lambda_{k-1}} \geq \frac{1}{\lambda_{k-1}2^k}.$$ 
Thus, we can bound
$$\begin{aligned}
&\frac{1}{\beta} \log \lp \frac{M}{L} \rp + C_1 + \frac{1}{\lambda_{k-1}2^k} \leq \rho_{\Omega}(z_0,x_0) + \rho_{\Omega}(x_{k-1},z) \leq \rho_{\Omega}(z_0,z) \\
&\leq \frac{1}{\beta} \log \lp \frac{M}{\delta_{\Omega}(z)} \rp +C.
\end{aligned}$$
This implies that
$$\log \lp \frac{1}{L \rho_{\Omega}(z)} \rp \leq 1+ \log \lp \frac{\delta_{\Omega}(z)}{L} \rp \leq 1 +\beta \lp C -C_1 - \frac{1}{\lambda_{k-1}2^k} \rp$$
for each $z \in \ell[x_k,z_1]$, so by the definition of $\lambda_k$ and the induction hypothesis, we obtain
\begin{equation} \label{lambda0}
\log(\lambda_k) \leq 1 + \beta \lp C - C_1 - \frac{1}{(2\lambda_0)^k} \rp.
\end{equation}
Recall that we chose $C_1$ large enough that $\lambda_0 < \beta/4$. It is straightforward to calculate that for any $k \geq 1$, the function $t \mapsto t^{k+1}e^{\beta/(2t)^k}$ is decreasing on the interval $(0,\beta/4)$. Thus,
$$\lambda_0^{k+1} e^{\beta/(2\lambda_0)^k} \geq \lp \frac{\beta}{4} \rp^{k+1} e^{\beta/(\beta/2)^k} \geq \frac{1}{2} \lp \frac{\beta}{4} \rp^{k+1} \lp \frac{2^k}{\beta^{k-1}} \rp^2 \geq \frac{\beta^2}{8}.$$
Consequently, if we also impose the condition that $C_1$ is large enough that 
$$e^{1+ \beta(C-C_1)} \leq \beta^2/8,$$
then \eqref{lambda0} gives
$$\lambda_k \leq e^{1+\beta(C-C_1)} e^{-\beta/(2\lambda_0)^k} \leq \frac{\beta^2}{8} \cdot e^{-\beta/(2\lambda_0)^k} \leq  \lambda_0^{k+1},$$
as claimed. This completes the induction.

To finish the proof, observe that $z_1 \in \ell[x_k,z_1]$ for each $k$, so 
$$0< \frac{1}{\rho_{\Omega}(z_1)} \leq L\lambda_k \leq L \lambda_0^{k+1}.$$
The desired contradiction comes from taking $k$ large enough.
\end{proof}

Before finishing this section, we return to the discussion at the end of the previous section. There we stated Theorem \ref{Johnprop}, which gives geometric criteria for a Loewner chain to have a $\Lip(1/2)$ driving term. Although this result is interesting in its own right, it also provides motivation for the following theorem, which we view as an analog of Theorem \ref{slit} outside of the ``simple curve" setting. The statement is technical, and we put off further discussion until Section \ref{secnonslit}. Here, we use the notation $\log^+(x) = \max \{\log(x),0 \}$ for $x \geq 0$.

\begin{theorem} \label{nonslit}
Let $K_t$ be a geometric Loewner chain, and let $\Omega_t = \Hy \backslash K_t$ be the complementary domains. Suppose that, for each $0 \leq s < t \leq T$, there are points $x_0 \in \Omega_s$ and $z_0 \in K_t \backslash K_s$ such that
\begin{enumerate}
\item[\textup{(i)}] $\diam_{\Omega_s}(\{x_0\} \cup K_t \backslash K_s) \leq C_0 \delta_{\Omega_s}(x_0)$,
\item[\textup{(ii)}] $\diam_{\Omega_s}(K_t \backslash K_s) \leq C_0 \delta_{\Omega_s}(z_0) \lp 1 + \log^+ \lp \frac{1}{\delta_{\Omega_s}(z_0)} \rp \rp,$
\item[\textup{(iii)}]  there is an $L$-John curve in $\Omega_s$ with tip $x_0$ and base-point $\infty$, and
\item[\textup{(iv)}] $\rho_{\Omega_s}(x_0,z_0) \leq \frac{1}{\beta} \log^+ \lp \frac{\diam_{\Omega_s}(K_t \backslash K_s)}{\delta_{\Omega_s}(z_0)} \rp + C_0.$
\end{enumerate}
Then this chain has a continuous driving term, $\lambda$, with 
$$| \lambda(s) - \lambda(t) | \leq C \sqrt{|s-t|} \lp \log (1/|s-t|) \rp^{1/\beta}$$
for all $|s-t| \leq 1/2$, where $C>0$ depends only on $L$, $\beta$, $C_0$, and $T$.
\end{theorem}

\section{Slit half-planes and the Loewner equation} \label{secslit}

A slit half-plane is a domain of the form $\Hy \backslash \gamma(0,\tau]$, where $\gamma \colon [0,\tau] \rightarrow \cl{\Hy}$ is a simple curve in $\cl{\Hy}$ that meets $\R$ only at $\gamma(0)$. After re-parameterizing $\gamma$ by half-plane capacity, one obtains a geometric Loewner chain $K_t = \gamma(0,t]$ for $0 \leq t \leq T$, where $T=\hcap(\gamma)/2$. In fact, this Loewner chain corresponds to a continuous driving function $\lambda$, which can be explicitly described. If $g_t \colon \Hy \backslash K_t \rightarrow \Hy$ is the hydrodynamic map, then $g_t$ extends continuously to the ``tip" $\gamma(t)$ of the slit and $\lambda_t = g_t(\gamma(t))$.

In general, a Loewner chain $K_t$ is said to be \ti{generated by a simple curve} if it arises from a slit half-plane. Such chains therefore correspond to a unique simple curve $\gamma$, parameterized by half-plane capacity, for which $K_t = \gamma(0,t]$ and $\gamma$ intersects $\R$ only at $\gamma(0)$. Moreover, the chain is associated to a continuous driving function. Recall from our earlier discussion that $\SLE_{\kappa}$ chains are, almost surely, generated by a simple curve when $\kappa \leq 4$. On the deterministic side, there are several nice examples of slit half-planes whose driving functions are computable (see, for example, \cite{KNK04, LR14}).

Our primary goal in this section is to prove Theorem \ref{slit}, which tells us that if a slit half-plane is a H\"older domain, then the associated driving function has a modulus of continuity similar to that of Brownian motion. More specifically, let us recall the full statement.

\begin{thmslit}
Let $K_t$ be a geometric Loewner chain, for $0 \leq t \leq T$, that is generated by a simple curve. If $\Hy \backslash K_T$ is a $(\beta,C_0)$-H\"older domain, then this chain has a driving term, $\lambda$, for which
$$|\lambda_s - \lambda_t| \leq C \sqrt{|s-t| \log(1/|s-t|)}$$
if $|s-t| \leq 1/2$. Here, $C>0$ depends only on $\beta$, $C_0$, and $T$.
\end{thmslit}

We will regularly use the notation $\Omega_t = \Hy \backslash K_t = \Hy \backslash \gamma(0,t]$ for the domains corresponding to the Loewner chain. Moreover, $f_t \colon \Hy \rightarrow \Omega_t$ and $g_t \colon \Omega_t \rightarrow \Hy$ will denote the hydrodynamically normalized conformal maps (which are inverses of each other). The proof of Theorem \ref{slit} consists essentially of three steps.
\begin{enumerate}
\item[(i)] Restriction property: Each $\Omega_t$ is a $(\beta,C_1)$-H\"older domain, with $C_1$ depending only on $\beta$, $C_0$, and $T$.
\item[(ii)] Sub-invariance property: For each $0 \leq s < t \leq T$, the domain $g_s(\Omega_t) \subset \Hy$ is a $(\beta/2,C_2)$-H\"older domain, with $C_2$ depending only on $\beta$, $C_0$, and $T$.
\item[(iii)] Half-plane capacity estimate: There is good control on the diameter of the transition hull $K_{s,t} = g_s(K_t \backslash K_s)$ in terms of the total area of Whitney squares that intersect it, or equivalently, in terms of $\hcap(K_{s,t}) = 2(t-s)$.
\end{enumerate}
It is a standard fact that the driving function of a Loewner chain satisfies the bound $|\lambda_s - \lambda_t | \lesssim \diam(K_{s,t})$, cf. \cite[Lemma 4.1]{LR12}. This, along with the estimate in (iii), will give the desired result.

We should remark that (ii) is an example of the so-called ``sub-invariance principle." Namely, if $g \colon D \rightarrow D'$ is a conformal map with $D'$ a disk or a half-plane, and if $E \subset D$ is nice, then $g(E)$ is also nice. See \cite[p. 51]{GeHag} and \cite[Section 6]{Hein89} for further discussion of the sub-invariance principle, where it is applied to quasi-disks and to John domains.

Steps (i) and (ii) will follow almost immediately from more general properties of H\"older domains that are contained in larger domains. We establish these as two separate lemmas. To set them up, let $\Omega = \Hy \backslash K$ and $\Omega' = \Hy \backslash K'$ be domains corresponding to hulls $K$ and $K'$, with $\emptyset \neq K' \subset K$, so that $\Omega \subset \Omega'$.

\begin{lemma} \label{dense}
Suppose that $\Omega$ is a $(\beta,C_0)$-H\"older domain and is dense in $\Omega'$. Then $\Omega'$ is a $(\beta,C)$-H\"older domain, where $C>0$ depends only on $\beta$, $C_0$, and $\diam(K)$.
\end{lemma}

\begin{lemma} \label{subinv}
Suppose that $\Omega$ is a $(\beta,C_0)$-H\"older domain and that $g \colon \Omega' \rightarrow \Hy$ is the hydrodynamic conformal map. Then $g(\Omega) \subset \Hy$ is a $(\beta/2,C)$-H\"older domain, where $C>0$ depends only on $\beta$, $C_0$, and $\diam(K)$.
\end{lemma}

\begin{proof}[Proof of Lemma \ref{dense}]
By hypothesis, $\Omega$ is a $(\beta, C_0)$ H\"older domain. Let $z_0 \in \Omega \subset \Omega'$ be the point coming from Lemma \ref{hypgrowth}. Recall that this point can be chosen so that $\dist(z_0,K) \leq 200\diam(K)$ and $\Im(z_0) \geq 100\diam(K)$. In particular, $\dist(z_0,K') \leq 300\diam(K)$ and $\Im(z_0) \geq 10\diam(K)$.

Fix a point $z \in \Omega'$ with $\dist(z,K') \leq 10\diam(K)$. We wish to establish appropriate bounds on $\rho_{\Omega'}(z_0,z)$. To this end, let $r = \delta_{\Omega'}(z) >0$ denote the distance from $z$ to $\partial \Omega'$. Suppose first that $|z-z_0| \geq r/2$. As $\Omega$ is dense in $\Omega'$, we can find $z_1 \in \Omega$ for which
$$|z-z_1| \leq \min \{ \diam(K), r/4 \}.$$
Observe, in particular, that
$$\dist(z_1,K) \leq \dist(z,K) + |z-z_1| \leq 10\diam(K) + \diam(K) \leq 100 \diam(K),$$
so $z_1 \in D_K$. Let $\ell$ denote the hyperbolic geodesic from $z_0$ to $z_1$ in $\Omega$, and let $x \in \ell$ be a point on $\ell$ with $|z-x| = r/2$. We then have
\begin{equation} \label{Omega'}
\rho_{\Omega'}(z_0,z) \leq \rho_{\Omega'}(z_0,x) + \rho_{\Omega'}(x,z) \leq \rho_{\Omega}(z_0,x) + 2,
\end{equation}
where the second inequality uses two ingredients: the fact that $x \in B(z,r/2)$ implies $\rho_{\Omega'}(x,z) \leq 2$, and the fact that $\Omega \subset \Omega'$ implies $\rho_{\Omega'} \leq \rho_{\Omega}$ by the Schwarz lemma.

As $\Omega$ is a $(\beta,C_0)$-H\"older domain and $z_1 \in D_K$, Lemma \ref{hypext} gives the bound
$$\rho_{\Omega}(z_0,x) \leq \frac{1}{\beta} \log \lp \frac{1}{\length(\ell[x,z_1])} \rp +C,$$
with $C$ depending only on $\beta$, $C_0$, and $\diam(K)$. Here, we have incorporated the quantity $\max \{ \diam(K)^{\beta}, \diam(K) \}$ into the constant $C$. This estimate, along with the bound in \eqref{Omega'} and the fact that $\length(\ell[x,z_1]) \geq r/4 = \delta_{\Omega'}(z)/4$, gives
$$\rho_{\Omega'}(z_0,z) \leq \frac{1}{\beta} \log \lp \frac{1}{\delta_{\Omega'}(z)} \rp + C',$$
with $C'$ depending only on $\beta$, $C_0$, and $\diam(K)$. 

If $|z-z_0| < r/2$ then the same inequality holds trivially, provided that $C'$ is large enough. Thus, $\Omega'$ with base-point $z_0$ satisfies the hypotheses in Lemma \ref{derivative} with $R=10\diam(K) \geq 10\diam(K')$. Consequently, $\Omega'$ is a $(\beta,C'')$-H\"older domain with $C''$ depending only on $\beta$, $C_0$, and $\diam(K)$.
\end{proof}

\begin{proof}[Proof of Lemma \ref{subinv}]
Again, by hypothesis, $\Omega$ is a $(\beta, C_0)$ H\"older domain. Let $z_0 \in \Omega$ be the point coming from Lemma \ref{hypgrowth} with $\dist(z_0,K) \leq 200\diam(K)$ and $\Im(z_0) \geq 100\diam(K)$. Recall from Lemma \ref{supp}(v) that $|g(z) - z| \leq 3\diam(K') \leq 3\diam(K)$ for all $z \in \Omega'$. In particular, if we define
$$\hat{K} = g(K \backslash K') = \Hy \backslash g(\Omega),$$
then $\hat{K}$ is contained in the $3\diam(K)$-neighborhood of $K$. Note that, as a result, $\diam(\hat{K}) \leq 7\diam(K)$. 

Let $w_0 = g(z_0)$ so that $\Im(w_0) \geq 90\diam(K)$ and $\dist(w_0,\hat{K}) \leq 300\diam(K)$. Our goal is to apply Lemma \ref{derivative} to the domain $g(\Omega) = \Hy \backslash \hat{K}$ with base-point $w_0$ and $R=70\diam(K) \geq 10\diam(\hat{K})$. To this end, fix $w \in g(\Omega)$ with $\dist(w,\hat{K}) \leq 10\diam(\hat{K})$. If $z = g^{-1}(w) \in \Omega$, then using Lemma \ref{supp}(v) once again, we have
$$\begin{aligned}
\dist(z,K) &\leq 3\diam(K') + \dist(w,K) \leq 6\diam(K) + \dist(w,\hat{K}) \\
&\leq 6\diam(K) + 70\diam(K) \leq 100\diam(K).
\end{aligned}$$
Thus, $z \in D_K$, so the H\"older property for $\Omega$ guarantees that
$$\rho_{g(\Omega)}(w_0,w) = \rho_{\Omega}(z_0,z) \leq \frac{1}{\beta} \log \lp \frac{1}{\delta_{\Omega}(z)} \rp + C,$$
where $C$ depends only on $\beta$, $C_0$, and $\diam(K)$.

We now apply Lemma \ref{neardist} to the map $f = g^{-1} \colon \Hy \rightarrow \Hy \backslash K'$ and the point $w \in \Hy$. More specifically, if $\Im(w) \leq 10\diam(K')$, then this lemma gives $\Im(w) \lesssim |f'(w)|$, with implicit constant depending only on $\diam(K')$, which is bounded above by $\diam(K)$. Otherwise, if $\Im(w) > 10\diam(K')$, then $|f'(w)| \approx 1$; using that $\Im(w) \leq 100\diam(K)$, we again have $\Im(w) \lesssim |f'(w)|$ with constant depending only on $\diam(K)$. Thus, in either case, the Koebe 1/4-theorem, implies that
$$\delta_{\Omega}(z) \geq \frac{1}{4}|f'(w)| \delta_{g(\Omega)}(w) \gtrsim \Im(w) \delta_{g(\Omega)}(w) \geq \delta_{g(\Omega)}(w)^2.$$
Inserting this into the previous inequality gives
$$\rho_{g(\Omega)}(w_0,w) \leq \frac{2}{\beta} \log \lp \frac{1}{\delta_{g(\Omega)}(w)} \rp + C',$$
where $C'$ depends only on $\beta$, $C_0$, and $\diam(K)$. Invoking Lemma \ref{derivative}, we conclude that $g(\Omega)$ is a $(\beta/2,C'')$-H\"older domain, where $C''$ depends only on $\beta$, $C_0$, and $\diam(K)$.
\end{proof}

The core of the proof of Theorem \ref{slit} consists in the following lemma, which controls the diameter of a simple curve in $\Hy$ whose complement is a H\"older domain. To put it in a clearer context, we recall that for any half-plane hull, $K$, the bound $\hcap(K) \lesssim \diam(K)^2$ holds, with an absolute constant. The lemma tells us that this is almost an equivalence (up to a log term) when $\Hy \backslash K$ has the H\"older property.

\begin{lemma} \label{Holddiam}
Let $K$ be a simple curve in $\cl{\Hy}$, intersecting $\R$ in exactly one point. If $\Hy \backslash K$ is a $(\beta,C_0)$-H\"older domain, then
$$\diam(K) \lesssim \sqrt{\hcap(K) \lp 1 + \log^+ \frac{1}{\diam(K)} \rp},$$
where $\log^+(x) = \max\{0,\log(x)\}$ and the implicit constant depends on $\beta$ and $C_0$.
\end{lemma}

\begin{proof}
Without loss of generality, we may assume that $K$ meets $\R$ at the point 0. As a first case, observe that if $\sup \{\Im(z) : z \in K\} \geq \sup \{|\Re(z)| : z \in K\}$, then we immediately have
$$\diam(K) \leq 4 \sup \{\Im(z) : z \in K\} \lesssim \sqrt{\hcap(K)}.$$
The last inequality here follows from Lemma \ref{Warea}, which says that $\hcap(K)$ is comparable to the total area of Whitney squares that intersect $K$. Thus, we may assume that $\sup \{\Im(z) : z \in K\} \leq \sup \{|\Re(z)| : z \in K\}$, so in particular, $\diam(K) \leq 4 \sup \{|\Re(z)| : z \in K\}$.

Let $\mathcal{W}$ denote the standard Whitney decomposition of $\Hy$, and let $\mathcal{C} \subset \mathcal{W}$ be the collection of all squares in $\mathcal{W}$ that intersect $K$ but have no ancestors (i.e., squares above it) that intersect $K$. For $Q \in \mathcal{C}$, let $\hat{Q}$ be the rectangle made up of $Q$ and all of its descendants:
$$\hat{Q} = \{z \in \Hy : z + iy \in Q \text{ for some } y \geq 0 \}.$$
Note that the collection $\{\hat{Q}: Q \in \mathcal{C} \}$ consists of rectangles with pairwise disjoint interiors, and the union of these rectangles contains $K$.

Let $w_0 \in K$ be a point with
$$\sup \{|\Re(z)| : z \in K\} \leq 2 |\Re(w_0)|,$$
and for which there is a square $Q_0 \in \mathcal{C}$ that contains $w_0$. For notational ease we will assume that $\Re(w_0) > 0$; the argument works in the same way if $\Re(w_0) < 0$. For $i \geq 1$, inductively define $Q_i$ to be the (unique) square in $\mathcal{C}$ for which
$$\max \{\Re(z) : z \in Q_i \} = \min\{ \Re(z) : z \in Q_{i-1} \}.$$
We do this for $i \leq n$, where $n$ is the first time at which 
$$\min \{ \Re(z) : z \in Q_n \} \leq \Re(w_0)/2.$$
Notice that for each $1\leq i \leq n$, the right edge of $\hat{Q}_i$ intersects the left edge of $\hat{Q}_{i-1}$. Moreover, the choice of $n$ ensures that $\max \{ \Re(z) : z \in Q_n \} \geq \Re(w_0)/2$.

Let $z_0 \in \Hy \backslash K$ and $C >0$ be the point and the constant given by Lemma \ref{hypgrowth}, so $C$ depends only on $\beta$ and $C_0$. Recall also that we may take
$$50\diam(K) \leq \delta_{\Hy \backslash K}(z_0) \leq \Im(z_0) \leq 200\diam(K).$$
Then choose $z_1 \in \Hy \backslash K$ with $0 < \Re(z_1) < \Re(w_0)/4$ very small, and for which $\Im(z_1)$ is small enough that the vertical segment from $z_1$ to $\R$ does not meet $K$. This is possible because $K$ is assumed to be a simple curve meeting $\R$ only at the point 0. Let $\ell$ be the hyperbolic geodesic in $\Hy \backslash K$ from $z_0$ to $z_1$.

The choice of $z_1$, along with the fact that $K$ is a simple curve and $z_0$ lies far above $K$, implies that $\ell$ must cross each of the rectangles $\hat{Q}_i$. More specifically, there exist points $x_i \in \ell$ that lie on the right edge of $\hat{Q}_i$ for each $1 \leq i \leq n$. Observe that the hyperbolic length of $\ell[x_i,x_{i+1}]$ is at least $1/4$. Indeed, this segment contains a sub-arc lying entirely in $\hat{Q}_i$ of length at least the side length of $Q_i$. Using Lemma \ref{hypext}, we have
$$\frac{n}{4} \leq \rho_{\Hy \backslash K}(z_0,x_n) \leq \frac{1}{\beta} \log \lp \frac{\max \{ \diam(K)^{\beta}, \diam(K) \}}{\length(\ell[x_n,z_1])} \rp+ C',$$
where $C' >0$ depends only on $\beta$ and $C_0$. Noting that $|x_n - z_1| \geq \Re(w_0)/4 \geq \diam(K)/40$, we obtain 
$$n \lesssim 1 + \log^+ \lp \frac{\max \{ \diam(K)^{\beta}, \diam(K) \}}{\diam(K)} \rp \lesssim 1+ \log^+ \lp \frac{1}{\diam(K)} \rp,$$
with constant depending on $\beta$ and $C_0$. By the Cauchy-Schwarz inequality, we have
$$\begin{aligned}
\diam(K) \lesssim \Re(w_0) &\lesssim \sum_{i=0}^n \diam(Q_i) \leq (n+1)^{1/2} \lp \sum_{i=0}^n \diam(Q_i)^2 \rp^{1/2}  \\
&\lesssim \sqrt{\hcap(K) \lp 1 + \log^+ \lp 1/\diam(K) \rp \rp},
\end{aligned}$$
where the final inequality is a consequence of Lemma \ref{Warea}.
\end{proof}

Putting together the three preceding lemmas, we are able to prove Theorem \ref{slit} without much difficulty. 

\begin{proof}[\textbf{Proof of Theorem \ref{slit}}]
By assumption, the Loewner chain $K_t$ is generated by a simple curve. Let $\gamma \colon [0,T] \rightarrow \cl{\Hy}$ be this curve, parameterized by half-plane capacity, so that $\gamma(0) \in \R$ and $K_t = \gamma(0,t]$ for each $t$. Let $\Omega_t = \Hy \backslash K_t$ denote the corresponding domains. By assumption, $\Omega_T$ is a $(\beta,C_0)$-H\"older domain.

Let us first make a remark about the dependence of the implicit constant on the parameter $T$. In a few places, we will need to allow certain constants to depend on $\diam(K_t)$, which is of course bounded from above by $\diam(K_T)$. Moreover, as $\Omega_T$ is a $(\beta,C_0)$-H\"older domain, Lemma \ref{Holddiam} implies that 
$$\diam(K_T) \lesssim \sqrt{ T \lp 1+ \log^+(1/\diam(K_T)) \rp} \lesssim \sqrt{ T \lp 1+ \log^+(1/T) \rp},$$
with implicit constants depending only on $\beta$ and $C_0$. Here, we have used that $\hcap(K_T) =2T$ and $\diam(K_T) \gtrsim \hcap(K_T)^{1/2}$; this latter inequality follows immediately from Lemma \ref{Warea}. Thus, $\diam(K_T)$ is bounded by an increasing function of $T$, so dependence on $\beta$, $C_0$, and $\diam(K_T)$ can be replaced by dependence on $\beta$, $C_0$, and $T$. 

We now proceed to the proof. Fix $0 \leq s < t \leq T$. As $\gamma$ is a simple curve, we know that $\Omega_T \subset \Omega_t$ is dense in $\Omega_t$. Thus, Lemma \ref{dense} implies that $\Omega_t$ is a $(\beta,C_1)$-H\"older domain with $C_1>0$ depending only on $\beta$, $C_0$, and $T$.

Let $g_s \colon \Omega_s \rightarrow \Hy$ be the hydrodynamic conformal map. Lemma \ref{subinv} guarantees that $g_s(\Omega_t)$ is a $(\beta/2,C_2)$-H\"older domain, where $C_2>0$ also depends only on $\beta$, $C_0$, and $T$. The corresponding transition hull for this domain,
$$K_{s,t} = g_s(K_t \backslash K_s) = \Hy \backslash g_s(\Omega_t),$$
is a simple curve that intersects $\R$ in exactly one point. Lemma \ref{Holddiam}, along with the facts that $\hcap(K_{s,t}) = 2(t-s)$ and $\diam(K_{s,t}) \gtrsim \hcap(K_{s,t})^{1/2}$, therefore gives
\begin{equation} \label{good}
\diam(K_{s,t}) \lesssim \sqrt{|s-t| \lp 1+ \log^+ (1/|s-t|) \rp},
\end{equation}
with implicit constant again depending only on $\beta$, $C_0$, and $T$.

It is a general fact that Loewner chains generated by a simple curve have continuous driving functions (cf. \cite[Chapter 4, Section 1]{Law05}). Alternatively, we could note that the right-hand side of \eqref{good} goes to 0 as $|s-t|$ goes to zero, so Pommerenke's theorem ensures that the chain $K_t$ has a continuous driving term. Regardless, let $\lambda$ denote the driving function. It is a standard fact (cf. \cite[Lemma 4.1]{LR12}) that $|\lambda_s - \lambda_t| \leq 4\diam(K_{s,t})$,
and invoking \eqref{good} once more, we obtain
$$|\lambda_s - \lambda_t| \lesssim \sqrt{|s-t| \log(1/|s-t|)}$$
for $|s-t|$ small enough; certainly $|s-t| \leq 1/2$ works.
\end{proof}

\section{Non-slit half-planes} \label{secnonslit}

In the previous section, we established that geometric Loewner chains corresponding to slit half-planes with the H\"older property have driving terms with regularity similar to that of Brownian motion. In an informal sense, we see this as a type of converse to the fact that $\SLE_{\kappa}$ domains are H\"older domains, for $\kappa < 4$. We mentioned earlier that when $\kappa > 4$, the $\SLE_{\kappa}$ domains are still H\"older domains, even though the corresponding curves are not simple. A natural question, then, is the following.

\begin{question} \label{q}
Is there a result that is, to the regime $\kappa > 4$, analogous to what Theorem \ref{slit} is to the regime $0 <\kappa<4$?
\end{question}

A first attempt toward such a statement might ask for a direct analogy: is Theorem \ref{slit} true (possibly with a worse bound on the modulus of continuity for $\lambda$) without the assumption that $K_T$ is a simple curve? Perhaps it is true if the ``simple curve" assumption is replaced by the hypothesis that each $\Hy \backslash K_t$ is a $(\beta,C_0)$-H\"older domain? Unfortunately, these questions do not have positive answers: it is possible for the hulls to grow in such a way that their complementary domains have nicely-accessible boundaries (say, are uniformly John domains) but $K_t \backslash K_s$ lies very close to $\partial (\Hy \backslash K_s)$ for a relatively large distance. One can construct such chains, for example, using space-filling curves.

Our main goal in this section is to prove Theorem \ref{nonslit}, which is a partial answer to Question \ref{q}. The conditions we impose on the geometric Loewner chain are chosen precisely to prevent the growth of $K_t \backslash K_s$ ``along" the boundary of $\Hy \backslash K_s$. Namely, our hypotheses ensure that $K_t \backslash K_s$ does not lie \ti{too close} to $\partial(\Hy \backslash K_s)$ for \ti{too large} a distance. In making these conditions precise, we are motivated by the geometry of H\"older domains (or rather, localized H\"older-type growth for the hyperbolic metric) rather than the H\"older condition itself. Let us restate the result here for convenience. 

\begin{thmnonslit}
Let $K_t$ be a geometric Loewner chain, and let $\Omega_t = \Hy \backslash K_t$ be the complementary domains. Suppose that, for each $0 \leq s < t \leq T$, there are points $x_0 \in \Omega_s$ and $z_0 \in K_t \backslash K_s$ such that
\begin{enumerate}
\item[\textup{(i)}] $\diam_{\Omega_s}(\{x_0\} \cup K_t \backslash K_s) \leq C_0 \delta_{\Omega_s}(x_0)$,
\item[\textup{(ii)}] $\diam_{\Omega_s}(K_t \backslash K_s) \leq C_0 \delta_{\Omega_s}(z_0) \lp 1 + \log^+ \lp \frac{1}{\delta_{\Omega_s}(z_0)} \rp \rp,$
\item[\textup{(iii)}] there is an $L$-John curve in $\Omega_s$ with tip $x_0$ and base-point $\infty$, and
\item[\textup{(iv)}] $\rho_{\Omega_s}(x_0,z_0) \leq \frac{1}{\beta} \log^+ \lp \frac{\diam_{\Omega_s}(K_t \backslash K_s)}{\delta_{\Omega_s}(z_0)} \rp + C_0.$
\end{enumerate}
Then this chain has a driving term, $\lambda$, with 
$$| \lambda(s) - \lambda(t) | \leq C \sqrt{|s-t|} \lp \log (1/|s-t|) \rp^{1/\beta}$$
for all $|s-t| \leq 1/2$, where $C>0$ depends only on $L$, $\beta$, $C_0$, and $T$.
\end{thmnonslit}

Before taking up the full proof of Theorem \ref{nonslit}, we turn our attention to Theorem \ref{Johnprop}. The conditions we impose in this latter result are similar in spirit to those found in the former result, with the exception that they are motivated by John-type properties instead of H\"older-type properties (and also that they are much simpler!). As we mentioned in Section \ref{intro}, this result gives criteria under which a geometric Loewner chain has a $\Lip(1/2)$ driving term. Let us recall it here as well.

\begin{thmJohnprop}
Let $K_t$ be a Loewner chain, and let $\Omega_t = \Hy \backslash K_t$ be the complementary domains. Suppose that, for each $0 \leq s < t \leq T$, there is a point $z_0 \in K_t \backslash K_s$ for which
\begin{enumerate}
\item[\textup{(i)}] $\diam_{\Omega_s}(K_t \backslash K_s) \leq C_0 \delta_{\Omega_s}(z_0),$ and
\item[\textup{(ii)}] there is an $L$-John curve in $\Omega_s$ with tip $z_0$ and base-point $\infty$.
\end{enumerate}
Then this chain has a driving term, $\lambda$, for which $\| \lambda \|_{1/2} \leq C$, where $C > 0$ depends only on $L$ and $C_0$.
\end{thmJohnprop}

The methods we use to prove Theorems \ref{Johnprop} and \ref{nonslit} have much in common with each other. Central to both is the use of modulus of path families: a conformal invariant and an important tool in the analysis of metric spaces. We shall need this concept only in the planar setting, so we leave the more general metric theory aside (see \cite[Chapter 7]{Hei01} for further discussion). 

Let $\Omega \subset \C$ be a domain, and let $\Gamma$ be a family of paths in $\Omega$. A Borel function $\rho \colon \Omega \rightarrow [0,\infty]$ is called a \ti{density}, and it is said to be \ti{admissible for $\Gamma$} if 
$$\int_{\gamma} \rho ds \geq 1$$
for each locally rectifiable path $\gamma \in \Gamma$, where $ds$ denotes integration with respect to arc-length. The \ti{modulus of $\Gamma$} is then defined as 
$$\mod_2(\Gamma,\Omega) = \inf_{\rho} \int_{\Omega} \rho^2 dA,$$
where the infimum is taken over all admissible densities for $\Omega$, and $dA$ denotes integration with respect to Lebesgue measure. Two basic properties of modulus that we will need are the following (cf. \cite[pp. 50--52]{Hei01}).

\begin{enumerate}
\item[(i)] $\mod_2$ is an outer measure on the collection of paths in $\Omega$: if $\Gamma_1, \Gamma_2, \ldots$ are families of paths in $\Omega$, then $\mod_2(\cup_i \Gamma_i, \Omega) \leq \sum_i \mod_2(\Gamma_i, \Omega)$.
\item[(ii)] $\mod_2$ is a conformal invariant: if $f \colon \Omega \rightarrow \Omega'$ is conformal and $\Gamma$ is a path family in $\Omega$, then $\mod_2(\Gamma, \Omega)  = \mod_2(\Gamma', \Omega')$, where $\Gamma' = \{f \circ \gamma : \gamma \in \Gamma \}$.
\end{enumerate}

Another tool that will be helpful for us is a topological result due to Janiszewski, which is often ``very useful for giving rigorous proofs of `obvious' facts in plane topology" \cite[p. 2]{Pom92}. Our application of it will not contradict this characteristic.

\begin{lemma}[Janiszewski]
Let $A$ be a closed set in $\cl{\D}$ such that $x,y \in \cl{\D}$ lie in different path components of $\cl{\D} \backslash A$. Then there is a connected component, $C$, of $A$ for which $x$ and $y$ lie in different path components of $\cl{\D} \backslash C$.
\end{lemma}

\begin{proof}
A more common version of Janiszewski's lemma deals with the case that $A$ has two connected components: if $A_1$ and $A_2$ are disjoint, closed subsets of $\cl{\D}$ such that $A_1 \cup A_2$ separates $x$ and $y$ in $\cl{\D}$, then either $A_1$ or $A_2$ must separate $x$ and $y$ (cf. \cite[Theorem 4.26]{Wil10}). In the terminology of \cite[Section II.4]{Wild}, this means that $\cl{\D}$ has the Phragman-Brouwer property. As $\cl{\D}$ is connected and locally connected, this is equivalent to the Brouwer property: if $M$ is a closed and connected subset of $\cl{\D}$ and $V$ is a connected component of $\cl{\D} \backslash M$, then $\partial V := \cl{V} \backslash V$ is closed and connected \cite[Theorem II.4.3]{Wild}.

Let $A$ be as in the statement of the lemma, and let $U$ be the connected component of $\cl{\D}\backslash A$ that contains $x$. In particular, $\cl{U}$ is closed and connected. Let $V$ be the connected component of $\cl{\D} \backslash \cl{U}$ that contains $y$, so the Brouwer property implies that $\partial V$ is connected. Note that $\partial V \subset \partial U \subset A$. Moreover, $\partial V$ separates $x$ and $y$ in $\cl{\D}$ because $x \notin \cl{V}$ and $y \in V$. Thus, the connected component of $A$ that contains $\partial V$ separates $x$ and $y$.
\end{proof}

The main ingredients for the proofs of Theorems \ref{Johnprop} and \ref{nonslit} are the same, so we begin with a few lemmas (likely familiar to many experts in this field) that will be used in both. Let us recall the notation $\log^+(x) = \max\{ \log(x),0\}$ for $x \geq 0$.

\begin{lemma} \label{Hy}
Let $z \in \Hy$ and let $\ell$ denote the vertical line from $0$ to $\infty$. Then
$$\log^+ \lp \frac{|\Re(z)|}{\Im(z)} \rp -2 \leq \rho_{\Hy}(z,\ell) \leq \log^+ \lp \frac{|\Re(z)|}{\Im(z)}\rp + 2.$$
\end{lemma}

\begin{proof}
By scaling, which is an isometry of $\Hy$ preserving $\ell$, we may assume that $|z|=1$. By reflection across $\ell$, we may also assume that $\Re(z) \geq 0$. As such, we can write $z= \cos(t_0) + i \sin(t_0)$ for some $0 < t_0 \leq \pi/2$. The hyperbolic projection of $z$ onto $\ell$ is $i$, so computing the hyperbolic distance gives
$$\rho_{\Hy}(z,\ell) = \rho_{\Hy}(z,i) = \int_{t_0}^{\pi/2} \frac{dt}{\sin t} = \log \frac{\cos(t_0/2)}{\sin(t_0/2)}.$$
If $t_0 \geq \pi/10$, the desired inequality holds easily by noting that 
$$0 \leq \rho_{\Hy}(z,\ell) \leq \log \frac{\cos(\pi/20)}{\sin(\pi/20)} \leq 2,$$
and
$$0 \leq \log^+ \lp \frac{|\Re(z)|}{\Im(z)} \rp = \log^+ \lp \frac{\cos(t_0)}{\sin(t_0)} \rp \leq \log \lp \frac{1}{\sin(\pi/10)}\rp \leq 2.$$ 
Suppose, then, that $0 < t_0 < \pi/10$. As $t \mapsto \cot(t)$ is decreasing on the interval $(0,\pi)$, we have
$$\rho_{\Hy}(z,\ell) = \log \lp \frac{\cos(t_0/2)}{\sin(t_0/2)} \rp \geq \log \lp \frac{\cos(t_0)}{\sin(t_0)} \rp = \log^+ \lp \frac{|\Re(z)|}{\Im(z)} \rp.$$
For the opposite inequality, observe that
$$\cos(t_0/2) \leq 2\cos(t_0) \hspace{0.3cm} \text{and} \hspace{0.3cm} \sin(t_0/2) \geq \sin(t_0)/2,$$
so
$$\rho_{\Hy}(z,\ell) \leq \log \lp \frac{4\cos(t_0)}{\sin(t_0)} \rp \leq \log^+ \lp \frac{|\Re(z)|}{\Im(z)} \rp+ 2,$$
as desired.
\end{proof}

\begin{lemma} \label{WhitneyballRay}
Let $\Omega \subsetneq \C$ be a simply connected domain, and let $\ell$ be a hyperbolic geodesic line in $\Omega$. Fix $z \in \Omega$ and let $\Gamma$ denote the family of paths in $\Omega$ that join the ball $B(z,\delta_{\Omega}(z)/2)$ to $\ell$. Then
$$\rho_{\Omega}(z,\ell) \leq \frac{\pi}{\mod_2(\Gamma,\Omega)} + 3.$$
\end{lemma}

We should remark here that if $z \in \ell$, then $\Gamma$ contains a constant path. As such, there is no admissible density $\rho$ for $\Gamma$, so $\mod_2(\Gamma,\Omega) = \infty$.

\begin{proof}
Let $p \in \ell$ be the point for which $\rho_{\Omega}(z,p) = \rho_{\Omega}(z,\ell)$, and let $f \colon \Omega \rightarrow \Hy$ be the conformal map for which $f(p)=i$ and $f(z)=ai$ for some $0 < a \leq 1$. The hyperbolic geodesic in $\Omega$ determined by $p$ and $z$ intersects $\ell$ orthogonally at $p$, so $\ell' = f(\ell)$ must intersect the imaginary axis orthogonally at $i$. Consequently, we know that $\ell'$ is the semi-circe $\partial \D \cap \Hy$.

As $f$ is an isometry in the hyperbolic metrics, we have 
$$\rho_{\Omega}(z,\ell) = \rho_{\Hy}(ai,\ell') = \rho_{\Hy}(ai,i) = \log(1/a).$$
Let $E = f(B(z,\delta_{\Omega}(z)/2))$, and let $\Gamma'$ denote the family of paths in $\Hy$ that join $E$ to the line $\ell'$. Conformal invariance of modulus implies that $\mod_2(\Gamma, \Omega) = \mod_2(\Gamma', \Hy)$. It therefore suffices to prove that
$$\log(1/a) \leq \frac{\pi}{\mod_2(\Gamma',\Hy)} + 3.$$

To this end, we first note that if $a \geq 1/10$, the desired inequality holds trivially. Thus, we may assume that $a < 1/10$. By the Koebe $1/4$-theorem, we have 
$$|f'(z)|\delta_{\Omega}(z) \leq 4 \delta_{\Hy}(f(z)) = 4a.$$ 
Moreover, the growth theorem \cite[Theorem 2.6]{Dur83} ensures that
$$|f(w)-f(z)| \leq 2|f'(z)|\delta_{\Omega}(z) \leq 8a$$
for all $w \in B(z,\delta_{\Omega}(z)/2)$, so that $E \subset B(0,10a)$. We now exhibit a density $\rho$ that gives the desired upper bound on $\mod_2(\Gamma',\Hy)$, namely,
$$\rho(z) = \frac{1}{\log(1/10a)|z|}\cdot \chi_{\{z \in \Hy: 10a \leq |z| \leq 1\}}(z), \hspace{0.3cm} \text{for }
z \in \Hy.$$
Every path in $\Gamma'$ joins $\{z \in \Hy : |z|=10a\}$ to $\ell' = \{z \in \Hy : |z|=1\}$, and this easily implies that $\rho$ is admissible for $\Gamma'$. Consequently,
$$\mod_2(\Gamma',\Hy) \leq \int_{\Hy} \rho(z)^2 dA = \frac{1}{\log(1/10a)^2} \int_0^{\pi} \int_{10a}^1 \frac{1}{r} dr d\theta = \frac{\pi}{\log(1/10a)}.$$
Rearranging this inequality gives
$$\log(1/a) \leq \frac{\pi}{\mod_2(\Gamma',\Hy)} + \log(10),$$
as desired.
\end{proof}

\begin{lemma} \label{CarrotRay}
Let $K$ be a half-plane hull, and let $\ell$ be a hyperbolic geodesic line in $\Omega = \Hy \backslash K$ with one endpoint at $\infty$. Suppose that there is an $L$-John curve in $\Omega$ with tip $z \in \Omega$ and base-point $\infty$. Then
$$\rho_{\Omega}(z,\ell) \leq C \lp 1+ \log^+ \lp \frac{\dist_{\Omega}(z,\ell)}{\delta_{\Omega}(z)} \rp \rp,$$
where $C>0$ depends only on $L$.
\end{lemma}

\begin{proof}
We may, of course, assume that $z$ does not lie on $\ell$. For notational ease, let $R = 2 \dist_{\Omega}(z,\ell) >0$. If $\delta_{\Omega}(z) > R$, then there is a point $w \in \ell$ with $|z-w| < \delta_{\Omega}(z)/2$, which immediately gives the bound $\rho_{\Omega}(z,\ell) \leq 2$. Thus, we may also assume that $\delta_{\Omega}(z) \leq R$.

Let $\alpha \subset \Omega$ be the $L$-John curve with tip $z$ and base-point at infinity. Suppose, as a first case, that $\alpha$ intersects $\ell$ inside the closed ball $\cl{B}(z,2R)$. Let $x \in \alpha \cap \ell$ be a point of intersection inside this ball, so Lemma \ref{Johncone}, applied to the curve $\alpha$, gives the bound
$$\rho_{\Omega}(z,\ell) \leq \rho_{\Omega}(z,x) \lesssim 1 +  \log^+ \lp \frac{\delta_{\Omega}(x)}{\delta_{\Omega}(z)} \rp,$$
where the implicit constant depends only on $L$. As 
$$\delta_{\Omega}(x) \leq |x-z| + \delta_{\Omega}(z) \leq 3R = 6\dist_{\Omega}(z,\ell),$$
we obtain
$$\rho_{\Omega}(z,\ell) \lesssim 1 + \log^+ \lp \frac{\dist_{\Omega}(z,\ell)}{\delta_{\Omega}(z)} \rp$$
with constant depending only on $L$.

We may now assume that $\alpha$ and $\ell$ are disjoint in $\cl{B}(z,2R)$. Our first goal is to show that for all $R < r < 2R$, there is a sub-arc of the circle $\partial B(z,r)$ that lies in $\Omega$ and joins $\ell$ to $\alpha$. To do this, fix an arc $\beta \subset \Omega$ that joins $z$ to $\ell$ and has $\diam(\beta) < 2\dist_{\Omega}(z,\ell) = R$. Note that $\beta \subset B(z,R)$. Consider the three curves $\ell$, $\alpha$, and $\beta$. It is not difficult to see that there are sub-arcs $\ell' \subset \ell$, $\alpha' \subset \alpha$, and $\beta' \subset \beta$ such that any two of these sub-arcs intersect in a single point (possibly at infinity) and whose concatenation forms a closed Jordan curve on the Riemann sphere. 

Let $\mathcal{C} \subset \Omega$ denote this Jordan curve (note that we do not include the possible point at infinity). By the Jordan curve theorem, $\C \backslash \mathcal{C}$ has two connected components. As $\partial \Omega$ is connected and is disjoint from $\mathcal{C}$, it lies entirely in one of these components. Let $\Delta$ denote the other component. It is then not hard to see that 
$$\cl{\Delta} = \Delta \cup \mathcal{C} \subset \Omega$$
is homeomorphic to a closed triangle (possibly with one vertex removed), where the three sides of the triangle correspond to the ``sides" $\ell'$, $\alpha'$, and $\beta'$ of $\cl{\Delta}$.

Let $x \in B(z,R)$ be the point at which $\ell'$ and $\beta'$ intersect, and let 
$$y \in \Omega \cup \{\infty\} \backslash B(z,2R)$$
be the point at which $\ell'$ and $\alpha'$ intersect. As $x,y \in \mathcal{C} \cup\{\infty\}$, it is clear that there are many curves in $\cl{\Delta} \cup \{y\}$ that join $x$ and $y$. Fix $R < r < 2R$, and observe that every such curve passes through $\partial B(z,r) \cap \cl{\Delta}$. In other words, the closed set $\partial B(z,r) \cap \cl{\Delta}$ separates the points $x$ and $y$ in the topological triangle $\cl{\Delta} \cup \{y\}$. By Janiszewski's lemma, there is a connected component of $\partial B(z,r) \cap \cl{\Delta}$, call it $C_r$, that separates $x$ and $y$ in this triangle. Observe that $C_r$ is a closed sub-arc of the circle $\partial B(z,r)$ and its endpoints lie on $\mathcal{C}$. In particular, its endpoints must be on $\ell' \cup \alpha'$, due to the fact that $\beta' \subset B(z,R)$. Notice that if both endpoints of $C_r$ were on $\ell'$, then it could not separate the vertices $x$ and $y$ in the topological triangle $\cl{\Delta} \cup \{y\}$. Similarly, it is not possible for both endpoints of $C_r$ to be on $\alpha'$. Thus, $C_r$ is a sub-arc of $\partial B(z,r)$, lying inside $\cl{\Delta} \subset \Omega$, that joins $\ell'$ and $\alpha'$.

Now, let $A(z,R,2R) = B(z,2R) \backslash \cl{B}(z,R)$ denote the annulus, and let $\Gamma$ be the family of paths in $\Omega$ that join $\alpha \cap A(z,R,2R)$ to $\ell$. Suppose that $\rho$ is an admissible density for $\Gamma$, so by what we have just shown, we know that
$$\int_{\Omega \cap \partial B(z,r)} \rho ds \geq 1$$
for each $R < r < 2R$. The Cauchy-Schwarz inequality then gives
$$\int_{\Omega \cap \partial B(z,r)} \rho^2 ds \geq \frac{1}{2\pi r},$$
so that
$$\int_{\Omega} \rho^2 dA \geq \int_{R}^{2R} \lp \int_{\partial B(z,r) \cap \Omega} \rho^2 ds \rp dr \geq \int_{R}^{2R} \frac{1}{2 \pi r} dr = \frac{\log2}{2\pi}.$$
In this way, we find that $\mod_2(\Gamma, \Omega) \geq 1/25$.

By the $L$-John property for $\alpha$, every point $w \in \alpha \cap A(z,R,2R)$ has $\delta_{\Omega}(w) \geq R/L$. It is therefore possible to find a finite set of points
$$w_1,\ldots,w_n \in \alpha \cap A(z, R, 2R),$$
with $n\geq 1$ depending only on $L$, such that the collection of balls 
$$B(w_1, \delta_{\Omega}(w_1)/2),\ldots, B(w_n, \delta_{\Omega}(w_n)/2)$$
covers $\alpha \cap A(z, R, 2R)$. For each $1 \leq i \leq n$, let $\Gamma_i$ be the family of paths in $\Omega$ that join $B(w_i,\delta_{\Omega}(w_i)/2)$ to the geodesic line $\ell$. Thus,
$$\Gamma \subset \bigcup_{i=1}^n \Gamma_i,$$
so by the sub-additivity property for modulus, we obtain
$$1/25 \leq \mod_2(\Gamma,\Omega) \leq \sum_{i=1}^n \mod_2(\Gamma_i,\Omega).$$
This implies that there is $1 \leq i_0 \leq n$ for which $\mod_2(\Gamma_{i_0},\Omega) \gtrsim 1$, and Lemma \ref{WhitneyballRay} then gives
\begin{equation} \label{w_i_0}
\rho_{\Omega}(w_{i_0},\ell) \lesssim 1,
\end{equation}
where the implicit constants depend only on $L$.

Invoking the $L$-John condition for $\alpha$, along with Lemma \ref{Johncone}, once more, we have the bound
$$\rho_{\Omega}(z,w_{i_0}) \lesssim 1+ \log^+ \lp \frac{\delta_{\Omega}(w_{i_0})}{\delta_{\Omega}(z)} \rp$$
with implicit constant again depending only on $L$. As $\delta_{\Omega}(w_{i_0}) \leq 2R + \delta_{\Omega}(z) \leq 3R$, this gives
$$\rho_{\Omega}(z,w_{i_0}) \lesssim 1+ \log^+ \lp \frac{R}{\delta_{\Omega}(z)} \rp \lesssim 1+ \log^+ \lp \frac{\dist_{\Omega}(z,\ell)}{\delta_{\Omega}(z)} \rp.$$
Along with the bound $\rho_{\Omega}(w_{i_0},\ell) \lesssim 1$ from \eqref{w_i_0}, we obtain the desired result.
\end{proof}

With the three preceding lemmas established, we are ready to prove Theorems \ref{Johnprop} and \ref{nonslit}. 

\begin{proof}[\textbf{Proof of Theorem \ref{Johnprop}}]
Fix $0\leq s < t \leq T$. Let $z_0 \in K_t \backslash K_s$ be a point with
$$\diam_{\Omega_s}(K_t \backslash K_s) \leq C_0 \delta_{\Omega_s}(z_0),$$
and for which there is an $L$-John curve in $\Omega_s$ with tip $z_0$ and base-point at infinity. Let $K_{s,t} = g_s(K_t \backslash K_s) \subset \Hy$ denote the transition hull, and let $w_0 = g_s(z_0)$. Our goal is to bound $\diam(K_{s,t})$ from above by something comparable to the quantity $\sup\{ \Im(w) : w \in K_{s,t} \}$.

To this end, fix $w \in K_{s,t}$ and consider $|\Re(w)-\Re(w_0)|$. Let $\ell'$ denote the hyperbolic geodesic line in $\Hy$ that passes through $w$ and $\infty$. Then $\ell = g_s^{-1}(\ell')$ is a hyperbolic geodesic in $\Omega_s$ that intersects $K_t \backslash K_s$ and has an endpoint at infinity. In particular, as $z_0 \in K_t \backslash K_s$, we have
$$\dist_{\Omega_s}(z_0,\ell) \leq \diam_{\Omega_s}(K_t \backslash K_s)  \leq C_0 \delta_{\Omega_s}(z_0).$$
Combining this with Lemmas \ref{Hy} and \ref{CarrotRay}, we obtain
$$\begin{aligned}
\log^+ \lp \frac{|\Re(w)-\Re(w_0)|}{\Im(w_0)} \rp &\leq \rho_{\Hy}(w_0,\ell') + 2 = \rho_{\Omega_s}(z_0,\ell) +2 \\
&\lesssim 1 + \log^+ \lp \frac{\dist_{\Omega_s}(z_0,\ell)}{\delta_{\Omega_s}(z_0)} \rp \lesssim 1,
\end{aligned}$$
where the implicit constants depend only on $L$ and $C_0$.
Thus, we find
$$|\Re(w) - \Re(w_0)| \lesssim \Im(w_0),$$
for any $w \in K_{s,t}$, with constant depending only on $L$ and $C_0$. This inequality easily implies that $\diam(K_{s,t}) \lesssim \sup \{\Im(w) : w \in K_{s,t}\}$. 

Recall that by Lemma \ref{Warea}, which relates half-plane capacity to the area of Whitney squares intersecting a given hull, we have
$$\sup \{\Im(w) : w \in K_{s,t}\} \lesssim \sqrt{\hcap(K_{s,t})} = \sqrt{2|s-t|},$$
with uniform constants. Thus, $\diam(K_{s,t}) \lesssim \sqrt{|s-t|}$ with constant depending only on $L$ and $C_0$. 

Pommerenke's theorem now guarantees that there is a continuous driving term, $\lambda$, associated to the Loewner chain $K_t$. The desired $\Lip(1/2)$ bound on $\lambda$ then follows, once again, from the fact that $|\lambda_s -\lambda_t| \leq 4\diam(K_{s,t})$.
\end{proof}

\begin{proof}[\textbf{Proof of Theorem \ref{nonslit}}]
Let us again make a comment about the dependence of the implicit constant on $T$. In the course of the proof, we will use an estimate that depends on $\diam(K_T)$. Fortunately, our hypotheses allow us to replace this by dependence on $T$. Indeed, let $z_0 \in K_T$ be a point with
$$\diam(K_T) \leq C_0 \Im(z_0) \lp 1+ \log^+ \lp \frac{1}{\Im(z_0)} \rp \rp,$$
which is guaranteed by assumption (ii) applied to $s=0$ and $t=T$. Note that $\Im(z_0) \lesssim \sqrt{\hcap(K_T)} \lesssim \sqrt{T}$, with absolute constants, so
$$\diam(K_T) \lesssim \sqrt{T} \lp 1+ \log^+ (1/T) \rp,$$
with implicit constant depending only on $C_0$.

We now begin the proof. Fix $0 \leq s < t \leq T$, and let $x_0 \in \Omega_s$ and $z_0 \in K_t \backslash K_s$ be points satisfying the hypotheses of the theorem. Once again, let $K_{s,t} = g_s(K_t\backslash K_s)$ be the transition hull, and let $w_0 = g_s(z_0)$. Fix a point $w \in K_{s,t}$. As in the previous proof, we wish to bound $|\Re(w)-\Re(w_0)|$ in terms of $\Im(w_0)$. 

Towards this end, let $\ell'$ denote the hyperbolic geodesic line in $\Hy$ that passes through $w$ and $\infty$. Then $\ell = g_s^{-1}(\ell')$ is a hyperbolic geodesic in $\Omega_s$ that intersects $K_t \backslash K_s$ and has another endpoint at infinity. In particular,
$$\dist_{\Omega_s}(x_0,\ell) \leq \diam_{\Omega_s}(\{x_0\} \cup K_t \backslash K_s) \leq C_0 \delta_{\Omega_s}(x_0).$$
Along with Lemma \ref{CarrotRay}, this bound allows us to estimate
$$\rho_{\Omega_s}(x_0,\ell) \lesssim 1+ \log^+ \lp \frac{\dist_{\Omega_s}(x_0,\ell)}{\delta_{\Omega_s}(x_0)} \rp \lesssim 1,$$
with implicit constants depending only on $L$ and $C_0$. Consequently, we have
$$\begin{aligned}
\rho_{\Omega_s}(z_0,\ell) &\leq \rho_{\Omega_s}(z_0,x_0) + \rho_{\Omega_s}(x_0,\ell) \\
&\leq \frac{1}{\beta} \log^+ \lp \frac{\diam_{\Omega_s}(K_t \backslash K_s)}{\delta_{\Omega_s}(z_0)} \rp + C \\
&\leq  \frac{1}{\beta} \log^+ \lp 1+\log^+ \lp \frac{1}{\delta_{\Omega_s}(z_0)} \rp \rp + C',
\end{aligned}$$
where $C$ and $C'$ depend only on $L$ and $C_0$.

Consider the quantity $\delta_{\Omega_s}(z_0)$. If $\Im(w_0) \leq 10\diam(K_s)$, then the distortion estimate in Lemma \ref{neardist} guarantees that $\delta_{\Omega_s}(z_0) \gtrsim \Im(w_0)^2$, with implicit constant depending only on $\diam(K_s)$, and therefore depending only on $T$. Otherwise, if $\Im(w_0) > 10\diam(K_s)$, then we have $\delta_{\Omega_s}(z_0) \gtrsim \Im(w_0)$ because $g_s$ does not move points more than distance $3\diam(K_s)$. In either case, we obtain
$$\rho_{\Omega_s}(z_0,\ell) \leq \frac{1}{\beta} \log^+ \lp 1+\log^+ \lp \frac{1}{\Im(w_0)} \rp \rp + C'',$$
where $C''$ now depends on $L$, $\beta$, $C_0$, and $T$. This bound, along with Lemma \ref{Hy}, gives
$$\begin{aligned}
\log^+ \lp \frac{|\Re(w) - \Re(w_0)|}{\Im(w_0)} \rp &\leq \rho_{\Hy}(w_0, \ell') + 2 = \rho_{\Omega_s}(z_0,\ell) +2\\
&\leq \frac{1}{\beta} \log^+ \lp 1+\log^+ \lp \frac{1}{\Im(w_0)} \rp \rp + C'' + 2.
\end{aligned}$$
Exponentiating, we obtain
$$|\Re(w) - \Re(w_0)| \lesssim \Im(w_0) \lp 1+ \log^+ \lp \frac{1}{\Im(w_0)} \rp \rp^{1/\beta},$$
where the implicit constant depends on $L$, $\beta$, $C_0$, and $T$. 

Recall that $w \in K_{s,t}$ was arbitrary. If it were the case that 
$$\diam(K_{s,t}) \leq 4\sup\{\Im(w) : w \in K_{s,t} \},$$
then we would immediately have 
$$\diam(K_{s,t}) \lesssim \sqrt{\hcap(K_{s,t})} \lesssim \sqrt{|s-t|}$$
with absolute implicit constants. Otherwise, we know that
$$\begin{aligned}
\diam(K_{s,t}) &\lesssim \sup \{ |\Re(w) - \Re(w_0)| : w \in K_{s,t} \} \\
&\lesssim \Im(w_0) \lp 1+ \log^+ \lp \frac{1}{\Im(w_0)} \rp \rp^{1/\beta}.
\end{aligned}$$
If $\Im(w_0) \geq 1$, then this bound, along with the fact that 
\begin{equation} \label{s-t}
\Im(w_0)^2 \lesssim \hcap(K_{s,t}) \lesssim |s-t|,
\end{equation}
gives $\diam(K_{s,t}) \lesssim \sqrt{|s-t|}$. In the case that $\Im(w_0) < 1$, the estimate in \eqref{s-t} still holds and allows us to bound
\begin{equation} \label{s-t2}
\diam(K_{s,t}) \lesssim \sqrt{|s-t|} \lp 1 + \log^+ \lp \frac{1}{|s-t|} \rp \rp^{1/\beta}.
\end{equation}
In all of these inequalities, the implicit constants depend only on $L$, $\beta$, $C_0$, and $T$.

To conclude, we once again invoke Pommerenke's theorem to obtain a continuous driving term, $\lambda$, associated to this Loewner chain. The desired control on $\lambda$ then follows from \eqref{s-t2}, together with with the estimate $|\lambda_s - \lambda_t| \leq 4\diam(K_{s,t})$, which we have used several times before.
\end{proof}

In addition to Question \ref{q}, which motivated this section, we end with a few other questions of a similar spirit. First, what ``local" geometric properties, perhaps along the lines of the hypotheses in Theorem \ref{nonslit}, do $\SLE_{\kappa}$ domains satisfy, in particular when $\kappa > 4$? We know, of course, that such domains are H\"older domains, almost surely, but this unfortunately does not guarantee any of the local properties assumed in our theorem. There are good known estimates on the number of times $\SLE_{\kappa}$ curves cross annuli (cf. \cite{Wern11, KS13}); perhaps this could be a place to begin.

Ultimately, from a deterministic point of view, it is highly desirable to understand the converse of what we have studied here: which driving functions $\lambda$ generate Loewner chains whose complementary domains are H\"older? All known deterministic conditions break down if $|| \lambda ||_{1/2} \geq 4$, and moreover, they guarantee much stronger geometric properties than the H\"older property. From what we have established here, such driving functions must lie in the weak-$\Lip(1/2)$ class. Going further: are there finer analytic properties that these functions must have?

\end{document}